\documentclass[a4paper,12pt]{amsart}

\usepackage{geometry}
\usepackage[comma]{natbib}
\usepackage[inline]{enumitem}
\usepackage{graphicx}
\usepackage[textfont=it,tableposition=above,font=small]{caption}		

\usepackage{framed}			
\usepackage[foot]{amsaddr}	
\usepackage{url}			

\usepackage{xargs}			
\usepackage{amssymb}		

\usepackage{silence}
\usepackage{todonotes}

\usepackage{pgffor}

\RequirePackage[colorlinks,citecolor=blue,urlcolor=blue]{hyperref}
\usepackage[norefs,nocites]{refcheck}
\usepackage{multirow}				
\usepackage{arydshln}				
\usepackage{psfrag}
\usepackage[symbol]{footmisc}
\usepackage{mathtools}				
\usepackage{booktabs}               
\usepackage{rotating}

\usepackage{color}
\usepackage{xcolor}
\usepackage{colortbl}

\newtheorem{theorem}{Theorem}
\newtheorem{lemma}[theorem]{Lemma}

\theoremstyle{definition}
\newtheorem*{definition}{Definition}
\newtheorem{example}{Example}

\newcommand{\Prob}[1][\Sph]{\mathcal P\left({#1}\right)}  
\newcommand{\Meas}[1][\Sph]{\mathcal M\left(#1\right)}      
\newcommand{\MeasS}[1][\Sph]{\mathcal S\left({#1}\right)}   
\newcommand{\PP}{\mathsf P}								
\newcommand{\R}{\mathbb R}								
\newcommand{\tr}{^\mathsf{T}}							
\newcommand{\ang}[1]{\left\langle #1 \right\rangle}     
\newcommand{\norm}[1]{\left\Vert #1 \right\Vert}        
\newcommand{\genhalf}{\mathcal H^*}
\newcommand{\half}{\mathcal H}
\newcommand{\flag}{\mathcal F}
\newcommand{\maxd}{\alpha^*}	       					
\newcommand{\median}{\mathfrak M}									
\newcommand{\Hauss}{\delta_H}           
\newcommand{\cl}[1]{\mathrm{cl}\left(#1\right)}         
\newcommand{\HD}{hD}	           		
\newcommand{\SHD}{shD}	
\newcommand{\AHD}{ahD}					
\newcommand{\tAHD}{\textbf{ahD}}        
\newcommand{\tHD}{\textbf{hD}}          
\newcommand{\AD}{aD}                    
\newcommand{\Sph}[1][d-1]{\mathbb{S}^{#1}}
\newcommand{\compl}{^\mathsf{c}}        
\newcommand{\Gplus}{\mathcal G}			
\newcommand{\rad}[1]{\mathrm{rad}\left(#1\right)}     

\title[Theoretical properties of angular halfspace depth]{Theoretical properties of angular halfspace depth}

\author{Stanislav Nagy$^{1}$}
\author{Petra Laketa$^{1}$}

\email{nagy@karlin.mff.cuni.cz}

\address{$^1$
	Faculty of Mathematics and Physics,
	Charles University, Prague,
	Czech Republic
}

\date{\today}

\begin{document}

\maketitle

\begin{abstract}
The angular halfspace depth (\tAHD{}) is a natural modification of the celebrated halfspace (or Tukey) depth to the setup of directional data. It allows us to define elements of nonparametric inference, such as the median, the inter-quantile regions, or the rank statistics, for datasets supported in the unit sphere. Despite being introduced in 1987, \tAHD{} has never received ample recognition in the literature, mainly due to the lack of efficient algorithms for its computation. With the recent progress on the computational front, \tAHD{} however exhibits the potential for developing viable nonparametric statistics techniques for directional datasets. In this paper, we thoroughly treat the theoretical properties of \tAHD{}. We show that similarly to the classical halfspace depth for multivariate data, also \tAHD{} satisfies many desirable properties of a statistical depth function. Further, we derive uniform continuity/consistency results for the associated set of directional medians, and the central regions of \tAHD{}, the latter representing a depth-based analogue of the quantiles for directional data.
\end{abstract}

\section{Nonparametrics of directional data and angular halfspace depth}

Directional data analysis concerns the statistics of datasets bound to lie on the unit sphere $\Sph = \left\{ x \in \R^d \colon \norm{x} = 1 \right\}$. Despite sharing similarities with multivariate statistical methods, the particular geometry of the sphere makes the analysis of directional data challenging. The sphere $\Sph$ is a symmetric compact manifold, which makes many statistical methods from $\R^d$ fruitless or suboptimal when applied in $\Sph$ directly \citep{Watson1983, Mardia_Jupp2000, Ley_Verdebout2017}.

We consider the nonparametric analysis for directional data and the concept of depth functions, a statistical tool that introduces elements of nonparametrics to multivariate or non-Euclidean spaces. In the past decades, depths have garnered great success in multivariate analysis \citep{Donoho_Gasko1992, Liu_etal1999, Zuo_Serfling2000, Mosler2013, Chernozhukov_etal2017, Mosler_Mozharovskyi2022}. A prime example of a depth in the Euclidean space $\R^d$ is the \emph{halfspace depth} (\tHD{}, also called \emph{Tukey depth}, or \emph{location depth}) that is defined for $x \in \R^d$ and a Borel probability measure $P$ on $\R^d$ by
    \begin{equation}    \label{eq: halfspace depth}
    \HD(x; P) = \inf \left\{ P\left( H_{x,u} \right) \colon u \in \Sph \right\},     
    \end{equation} 
where $H_{x,u} = \left\{ y \in \R^d \colon \left\langle y, u \right\rangle \geq \left\langle x, u \right\rangle \right\}$ is the closed halfspace whose boundary passes through $x$ with inner unit normal $u$. In words, \tHD{} evaluates the smallest $P$-mass of a halfspace that contains $x$. As such, it allows to devise a $P$-dependent ordering of the points from $\R^d$ --- a point $x$ is said to be located deeper than $y$ inside the mass of $P$ if $\HD(x; P) > \HD(y; P)$. The deepest point in $\R^d$, defined as (any) point that maximizes the depth function $x \mapsto \HD(x; P)$ over $\R^d$, is a natural analogue of the median for $\R^d$-valued measures and is called a \emph{halfspace median} of $P$. A counterpart of quantiles (or, more precisely, inter-quantile regions) in $\R^d$ are the \emph{central regions} of $P$, given as the upper level sets of \tHD{}
    \begin{equation}    \label{eq: central regions}
    \HD_\alpha(P) = \left\{ x \in \R^d \colon \HD(x; P) \geq \alpha \right\} \quad \mbox{for }\alpha \geq 0.   
    \end{equation}  
These sets are known to be nested, convex, and compact for $\alpha > 0$; their shapes capture the geometry of $P$. Both the multivariate medians and central regions are of great importance in nonparametric analysis of multivariate data, and have been studied extensively in the literature \citep{Donoho_Gasko1992, Rousseeuw_Ruts1999, Zuo_Serfling2000, Dyckerhoff2017}. The halfspace depth is only one of many depth functions proposed in multivariate analysis. It is, nevertheless, the classical representative of multivariate depths. 

Several depths suitable specifically for directional data have been proposed in the literature; we refer to \citet{Liu_Singh1992, Agostinelli_Romanazzi2013b, Ley_etal2014, Pandolfo_etal2018b, Buttarazzi_etal2018, Konen2022}, and \citet{Hallin_etal2022}. In this contribution, we scrutinize historically the first directional depth function proposed by \citet[Example~2.3.4]{Small1987}. It is a version of \tHD{} from~\eqref{eq: halfspace depth} suitable for measures in $\Sph$. For $x \in \Sph$ and $P$ a Borel probability measure on $\Sph$, the \emph{angular halfspace depth} (\tAHD{}, also known as \emph{angular Tukey depth}) of $x$ with respect to (w.r.t.) $P$ is defined as
    \begin{equation}	\label{eq: angular halfspace depth}
    \AHD\left(x; P\right) = \inf \left\{ P\left( H_{0,u} \right) \colon u \in \Sph \mbox{ and } x \in H_{0,u} \right\}. 
    \end{equation}
In contrast to the standard \tHD{} in~\eqref{eq: halfspace depth}, in the definition of \tAHD{} one considers only halfspaces whose boundary passes through the origin $0$ in the ambient space $\R^d$, and searches for a minimum $P$-mass among those that contain $x$. 

The first rigorous studies of \tAHD{} were conducted by \citet{Small1987} and \citet[Section~4]{Liu_Singh1992}. Since then, no systematic investigation of the theory of \tAHD{} has been performed. In fact, it might come as a surprise how little attention did \tAHD{} receive in the literature, especially when compared with the abundant body of research on the classical \tHD{}. One explanation for this is the presumed high computational cost of the angular depth, coupled with a lack of efficient algorithms for its computation \citep{Pandolfo_etal2018b}. The problem of exact and approximate computation of \tAHD{} was, however, recently resolved in \citet{Dyckerhoff_Nagy2023}, which paved the way to explore the general theory and practice of \tAHD{} with its statistical applications. 

This paper comprehensively studies the main theoretical properties of \tAHD{}, its associated median, and its central regions. We contrast \tAHD{} with \tHD{} and demonstrate that similarly to the halfspace depth in $\R^d$, also the angular depth satisfies an array of plausible properties required from a proper depth function. After introducing the notations, we begin in Section~\ref{section: projection} by drawing a direct relation between \tAHD{} on the sphere $\Sph$ and a variant of \tHD{} in $\R^{d-1}$. Then, in Section~\ref{section: properties}, we study \tAHD{} w.r.t. the desirable properties of a directional depth formulated recently in \citet{Nagy_etal2023}. We show that \tAHD{}, as the only directional depth function found in the literature, satisfies all the conditions from \citet{Nagy_etal2023}. In the final Section~\ref{section: additional properties} we provide a list of additional characteristics of \tAHD{}, mainly related to the continuity of its median and the associated central regions. Those findings are important in statistical practice, as they guarantee the uniform consistency of the sample depth~\eqref{eq: angular halfspace depth} when computed w.r.t. the empirical measure of a random sample $X_1, \dots, X_n$ from distribution $P$. The proofs of all theoretical results are gathered in the Appendix.

\subsection*{Notations}

We write $e_j \in \Sph$ for the $j$-th canonical vector in $\R^d$, $j=1,\dots,d$. That is, $e_1 = (1, 0, \dots, 0)\tr$ etc. We denote by $\half = \left\{ H_{x, u} \colon x \in \R^d \mbox{ and }u \in \Sph \right\}$ the collection of all halfspaces in $\R^d$, and by $\half_0 = \left\{ H_{0,u} \in \half \colon u \in \Sph \right\}$ those halfspaces in $\R^d$ whose boundary passes through the origin. For a set $A \subset \R^d$, we write $A^\circ$ for the interior of $A$ and $\partial A$ for the topological boundary of $A$. Sometimes we use $A^\circ$ ($\partial A$) to denote the relative interior (boundary) of $A$, that is the interior (boundary) considered in the smallest affine subspace containing $A$. Whether we consider relative interior (boundary) or not will always be clear from the context. The complement of $A$ is $A\compl$. A set $A \subseteq \Sph$ is called \emph{spherical convex} \citep[see, e.g.,][]{Besau_Werner2016} if its radial extension defined as $\rad{A} = \left\{ \lambda \, a \in \R^d \colon a \in A \mbox{ and }\lambda > 0\right\}$ is convex in $\R^d$. We denote by $\Sph_+ = \left\{ y \in \Sph \colon \ang{y, e_d} > 0 \right\}$ and $\Sph_- = \left\{ y \in \Sph \colon \ang{y, e_d} < 0 \right\}$ the (open) ``northern'' hemisphere with a pole at $e_d$, and the (open) ``southern'' hemisphere centered at $-e_d$, respectively. The set $\Sph_0 = \left\{ y \in \Sph \colon \ang{y, e_d} = 0 \right\}$ is called the ``equator'' of the $(d-1)$-sphere $\Sph$.

Let $\left(\Omega, \mathcal A, \PP\right)$ be the probability space on which all random variables are defined. For a topological space $\mathcal X$, $\Prob[\mathcal X]$ stands for the collection of all Borel probability measures on $\mathcal X$, and $X \sim P \in \Prob[\mathcal X]$ means that $X$ is a random variable in $\mathcal X$ with distribution $P$. For $\varphi \colon \mathcal X \to \mathcal Y$ a map between topological spaces, we write $P_{\varphi(X)} \in \Prob[\mathcal Y]$ for the distribution of $\varphi(X)$ with $X \sim P \in \Prob[\mathcal X]$. Further, $\Meas[\mathcal X]$ represents the collection of all finite Borel measures on $\mathcal X$. Certainly, $\Prob[\mathcal X] \subset \Meas[\mathcal X]$. Weak convergence of a sequence of measures $\left\{ P_{n} \right\}_{n=1}^\infty \subset \Meas[\mathcal X]$ to $P \in \Meas[\mathcal X]$ is denoted by $P_n \xrightarrow{w} P$ as $n \to \infty$. Finally, we say that $Q$ is a (finite Borel) \emph{signed measure} on a topological space $\mathcal X$ if there exist two finite Borel measures $Q_+, Q_- \in \Meas[\mathcal X]$ such that \citep[Theorem~5.6.1]{Dudley2002}
    \begin{equation}    \label{eq: Hahn decomposition}
    Q(B) = Q_+(B) - Q_-(B)  \quad\mbox{for all $B \subseteq \mathcal X$ Borel}.  
    \end{equation}
A signed measure can attain both positive and negative values. We denote by $\MeasS[\mathcal X]$ the set of all finite Borel signed measures on $\mathcal X$, and note that $\Meas[\mathcal X] \subset \MeasS[\mathcal X]$.

%
%
%
%

\section{Gnomonic projection and \texorpdfstring{$\AHD$}{ahD}}  \label{section: projection}

We begin our study by drawing connections of \tAHD{} in $\Sph$ with the standard (Euclidean) \tHD{} in $\R^{d-1}$. They will allow us to use the abundance of theoretical results available for \tHD{} and adapt them to the setup of directional measures. First, note that the halfspace depth~\eqref{eq: halfspace depth} is well defined not only for probability measures but for any (finite) Borel measures $P \in \Meas[\R^{d}]$. It can be written in two equivalent forms, either as
    \begin{equation}	\label{eq: halfspace depth first expression}
    \HD\left(x;P\right) = \inf \left\{ P\left(H \right) \colon H \in \half \mbox{ and } x \in H \right\},
    \end{equation}
or also as
    \begin{equation} \label{eq: halfspace depth second expression}
    \HD\left(x;P\right) = \inf \left\{ P\left(H \right) \colon H \in \mathcal H \mbox{ and } x \in \partial H \right\}.    
    \end{equation}
In what follows, we employ the halfspace depth with signed measures. In case when $Q \in \MeasS[\R^d]$ can attain negative values, the two definitions of \tHD{} in~\eqref{eq: halfspace depth first expression} and~\eqref{eq: halfspace depth second expression} differ; a halfspace $H$ containing $x$ in its interior may have strictly smaller $Q$-mass than its subset $H'$ with $x \in \partial H'$. Out of the two possibilities of defining \tHD{} for signed measures, it will be convenient to use the first one in~\eqref{eq: halfspace depth first expression}. We say that the \emph{halfspace depth} of $x \in \R^d$ w.r.t. $Q \in \MeasS[\R^d]$ is defined as
    \begin{equation}    \label{eq: signed halfspace depth}
    \HD\left(x; Q\right) = \inf \left\{ Q\left(H \right) \colon H \in \half \mbox{ and } x \in H \right\}.    
    \end{equation}
We use the same notation for both halfspace depth of measures from $\Meas[\R^d]$ and signed measures from $\MeasS[\R^d]$; the standard depth~\eqref{eq: halfspace depth first expression} is only a particular case of our more general depth~\eqref{eq: signed halfspace depth}. 

Our task is to express \tAHD{} of $x \in \Sph$ w.r.t. $P \in \Prob$.\footnote{We could equally work with $P \in \Meas$ without having to restrict to probability measures. All our results also work for $P \in \Meas$, with obvious minor modifications.} Throughout this section, we make the following assumption on $P$ and $x$
    \begin{equation}    \label{eq: zero equator}
    P(\Sph_0) = 0 \quad \mbox{and}\quad x \in \Sph_+.
    \end{equation}
This assumption is made without loss of generality due to the rotational invariance of \tAHD{} that will be formally proved as Theorem~\ref{theorem: rotational} in Section~\ref{section: rotational}.\footnote{Indeed, if \eqref{eq: zero equator} is not valid for $Q \in \Prob$ and we want to compute $\AHD(z;Q)$ for $z \in \Sph$, it is always possible to find a direction $u \in \Sph$ such that $Q(\left\{y \in \Sph \colon \ang{y, u} = 0 \right\}) = 0$ and $\ang{y,z} > 0$. Applying (any) orthogonal rotation $O \in \R^{d \times d}$ such that $O\, e_d = u$ to $X \sim Q$, we obtain $P \in \Prob$, with $P$ the measure corresponding to the random vector $O X$. Then, \eqref{eq: zero equator} is true, and $x = O z \in \Sph_+$. By Theorem~\ref{theorem: rotational} we have $\AHD(x; P) = \AHD(z; Q)$, while the conditions~\eqref{eq: zero equator} are now valid.} Denote by $\Gplus = \left\{ y \in \R^d \colon \ang{y,e_d} = 1 \right\}$ the hyperplane tangent to $\Sph$ at the pole $e_d \in \R^d$, see Figure~\ref{figure: mapping}. Consider the mapping 
    \begin{equation}    \label{eq: gnomonic projection}
    \xi \colon \Sph \setminus \Sph_0 \to \Gplus \colon x \mapsto x/\ang{x,e_d}    
    \end{equation} 
that takes points $x \in \Sph$ not on the equator to the unique point of $\Gplus$ on the straight line between $x$ and the origin. This map is visualized for $d=2$ and $d=3$ in Figure~\ref{figure: mapping}. In what follows, we will canonically identify the hyperplane $\Gplus$ with $\R^{d-1}$ by formally dropping the last coordinate $\ang{y, e_d} = 1$ of points $y \in \Gplus$. This allows us to write also $\xi(x) \in \R^{d-1}$ for any $x \in \Sph \setminus \Sph_0$. The map~\eqref{eq: gnomonic projection} is called the \emph{gnomonic projection} of the sphere $\Sph$ into $\Gplus$ (or $\R^{d-1}$, see \citealp{Besau_Werner2016}). It is a double covering of $\R^{d-1}$, once for points from $\Sph_+$ and once for those from $\Sph_-$. On $\Sph_+$ (or $\Sph_-$) it is a bijection with $\R^{d-1}$. The points from the equator $\Sph_0$ are not considered in~\eqref{eq: gnomonic projection}. This will not be a problem because of our assumption~\eqref{eq: zero equator}. 

\begin{figure}[htpb]
\includegraphics[width=0.38\textwidth]{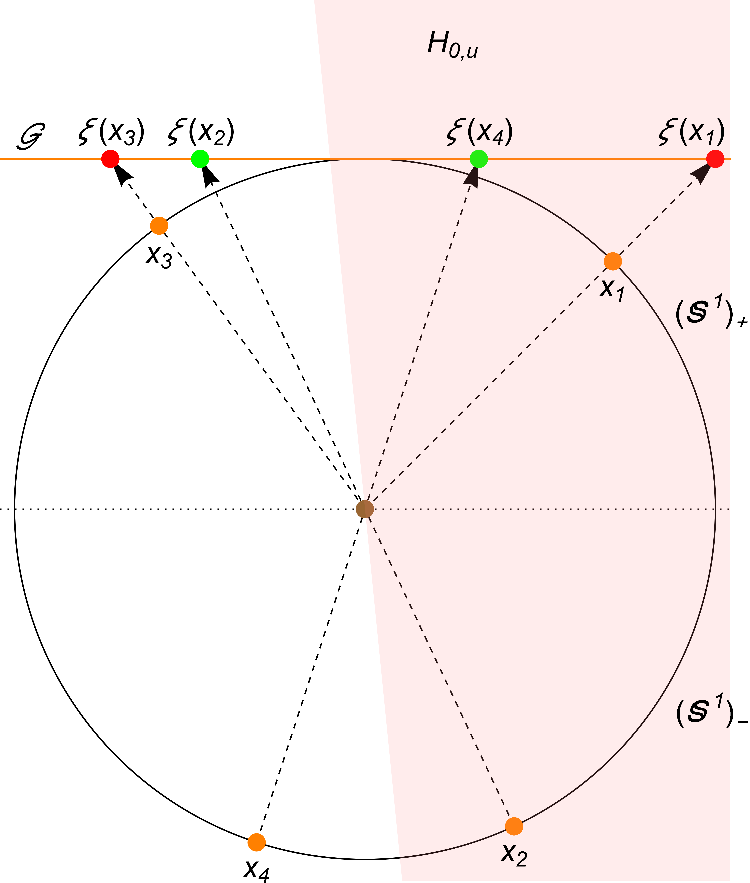} \qquad \includegraphics[width=0.53\textwidth]{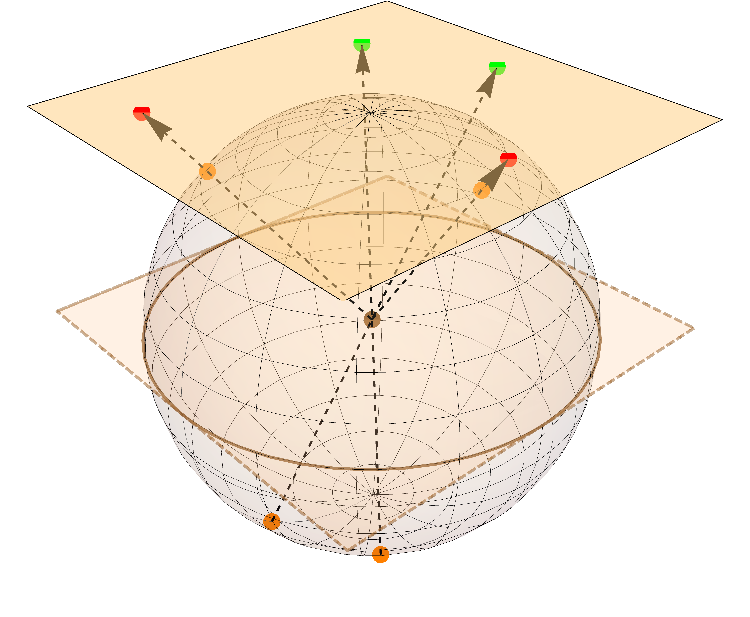}
\caption{Transformation $\xi$ that takes $\Sph\setminus\Sph_0$ to $\Gplus$ for $d=2$ (left panel) and $d=3$ (right panel). The data from the northern hemisphere $\Sph_+$ are mapped into points of positive $P_\pm$-mass in $\Gplus$ (red points), while data from the southern hemisphere $\Sph_-$ project to points of negative $P_\pm$-mass (green points). In the left panel we see also a halfspace $H_{0,u} \in \half_0$ (shaded halfplane). Only points $x_1 \in \Sph[1]_+$ and $x_2 \in \Sph[1]_-$ lie in $H_{0,u}$, which means that precisely $\xi(x_1)$ and $\xi(x_4)$ are contained in $G_u = H_{0,u} \cap \Gplus$. This is in accordance with formula~\eqref{eq: pm}.}
\label{figure: mapping}
\end{figure}

Denote for any $B \subseteq \Gplus$ (or $B \subseteq \R^{d-1}$) Borel
    \begin{equation*}    
    \begin{aligned}
    P_+(B) & = P\left( \left\{ y \in \Sph_+ \colon \xi(y) \in B \right\} \right), & & P_-(B) & = P\left( \left\{ y \in \Sph_- \colon \xi(y) \in B \right\} \right),
    \end{aligned}
    \end{equation*}
the pushforward measures of restrictions of $P$ to $\Sph_+$ and $\Sph_-$, respectively, under the gnomonic projection $\xi$. In words, $P_+$ is the image measure of the part of $P$ in the northern hemisphere when projected to $\Gplus$ (or $\R^{d-1}$), and $P_-$ analogously for the southern hemisphere. We patch $P_+$ and $P_-$ together and use them to define a signed measure $P_\pm \in \MeasS[\R^{d-1}]$ given by~\eqref{eq: Hahn decomposition} 
    \begin{equation}    \label{eq: gnomonic measure}
    P_{\pm}(B) = P\left( \left\{ x \in \Sph_+ \colon \xi(x) \in B \right\} \right) - P\left( \left\{ x \in \Sph_- \colon \xi(x) \in B \right\} \right) \quad \mbox{for $B \subseteq \R^{d-1}$ Borel}.	
    \end{equation}
The signed measure $P_\pm$ takes a simple form when $P$ is an empirical measure of $n$ data points $x_1, \dots, x_n \in \Sph$. In that situation, $P_{\pm}$ is simply a signed empirical measure, with atoms of positive mass $1/n$ at each $\xi(x_i)$ such that $x_i \in \Sph_+$, and atoms of negative mass $-1/n$ at $\xi(x_i)$ for $x_i \in \Sph_-$. In this situation, the gnomonic projection was essential to design fast computational algorithms for the sample \tAHD{} in \citet{Dyckerhoff_Nagy2023}.

Our principal tool for analyzing \tAHD{} is its relation with \tHD{}, which is described in the following theorem. In that result, we say that a set $H$ is called a \emph{generalized halfspace} in $\R^{d-1}$ if $H$ is a halfspace in $\R^{d-1}$, an empty set, or the whole space $\R^{d-1}$ itself. 

\begin{theorem}	\label{theorem: relation}
Let $P \in \Prob$ and $x \in \Sph$ be such that~\eqref{eq: zero equator} is true. Then
    \begin{equation}	\label{eq: general AHD formula}
    \AHD(x;P) = P\left(\Sph_-\right) + \inf \left\{ P_+\left( H \right) - P_-\left( H^\circ \right) \colon H \in \mathcal H^* \mbox{ and } \xi(x) \in H \right\}	
    \end{equation}
for $\genhalf$ the set of all generalized halfspaces in $\R^{d-1}$. If, in addition, 
    \begin{equation}    
    P\left( \partial H \cap \Sph_- \right) = 0 \quad \mbox{for all }H \in \half_0,	\label{eq: negative smoothness}  
    \end{equation}
then 
    \[  \AHD(x;P) = P(\Sph_-) + \HD(\xi(x); P_{\pm}).   \]
\end{theorem}

Formula~\eqref{eq: general AHD formula} will be quite useful in deriving theoretical properties of \tAHD{} and constructing examples, as we will see throughout this paper.

\section{Desiderata for a directional depth function}   \label{section: properties}

There are several well-established properties that a depth $D$ in a linear space $\R^d$ should obey. It is typically agreed that $D$ should be \begin{itemize*} \item affine invariant, meaning that it does not depend on the coordinate system in $\R^d$; \item maximized at the center of symmetry of $P$ for any $P \in \Prob[\R^d]$ symmetric; \item monotonically decreasing as $x$ moves along straight lines starting at the point of maximum of $D$; and \item decaying uniformly to zero as $\norm{x} \to \infty$. \end{itemize*} This set of properties was postulated by \citet{Zuo_Serfling2000}; for additional related axioms we refer to \citet{Liu1990, Serfling2006, Mosler2013}, and \citet{Mosler_Mozharovskyi2022}. 

Compared to depths in $\R^d$, much less is known about general depth functions defined in $\Sph$. One set of assumptions has been laid down recently in \citet{Nagy_etal2023}. There, a general \emph{angular} (or \emph{directional}) \emph{depth function} $\AD$ was introduced as a bounded map $\AD \colon \Sph \times \Prob[\Sph] \to [0,\infty)$ that fulfills (most of) the following properties for all $P \in \Prob[\Sph]$:
    \begin{enumerate}[label=\upshape{($\mathrm{D}_{\arabic*}$)}, ref=\upshape{($\mathrm{D}_{\arabic*}$)}]
    \item \label{cond: rotational} \emph{Rotational invariance:} $\AD(x;P)=\AD(O x; P_{OX})$ for all $x \in \Sph$ and any orthogonal matrix~$O \in \R^{d \times d}$, where $P_{OX} \in \Prob[\Sph]$ is the distribution of the transformed random vector $OX$ with $X \sim P$;
    \item \label{cond: spherical maximality} \emph{Maximality at center:} For any $P \in \Prob[\Sph]$ rotationally symmetric around the axis given by $\mu \in \Sph$ we have 
        \begin{equation}\label{eq: double maximality sphere}
        \max\{ \AD(\mu;P), \AD(-\mu;P) \} = \sup_{x \in \Sph} \AD(x;P);    
        \end{equation}
    \item \label{cond: spherical monotonicity} 
    \emph{Monotonicity along great circles:} 
        \begin{equation*}
        \AD(x; P) \leq \AD(\left(\mu + \alpha(x - \mu)\right)/\norm{ \mu + \alpha(x-\mu)};P)
        \end{equation*}
    for all $x \in \Sph \setminus\left\{-\mu\right\}$ and $\alpha \in [0,1]$, where $\mu \in \Sph$ is any point that satisfies
        \begin{equation}    \label{eq: maximality sphere}
        \AD(\mu;P) = \sup_{x \in \Sph} \AD(x;P); 
        \end{equation}
    \item \label{cond: minimality} \emph{Minimality at the anti-median:} $\AD(-\mu;P)=\inf_{x \in \Sph} \AD(x;P)$, for any $\mu \in \Sph$ that satisfies~\eqref{eq: maximality sphere}.
    \item \label{cond: spherical usc} 
    \emph{Upper semi-continuity:} $\AD(\cdot;P) \colon \Sph \to [0,\infty) \colon x \mapsto \AD(x;P)$ is upper semi-continuous, meaning that
        \[  {\lim\sup}_{x_n \to x} \AD(x_n;P) \leq \AD(x;P) \quad \mbox{for all }x \in \Sph \]
    where the sequence $\left\{x_n\right\}_{n=1}^\infty$ is taken in $\Sph$.
    \item \label{cond: spherical convexity} \emph{Quasi-concavity:} All central regions
        \begin{equation}    \label{eq:central region}
        \AD_\alpha(P) = \left\{ x \in \Sph \colon \AD(x;P) \geq \alpha \right\} \quad \mbox{with $\alpha \geq 0$}  
        \end{equation}
    are spherical convex sets.
    \item \label{cond: nonrigidity} \emph{Non-rigidity of central regions:} There exists a measure $P \in \Prob[\Sph]$ such that for some $\alpha > 0$ the central region from~\eqref{eq:central region} of $\AD$ is not a spherical cap.
    \end{enumerate}
Conditions~\ref{cond: rotational}--\ref{cond: minimality} are direct translations of the classical requirements~\textbf{P1}--\textbf{P4} postulated for the (Euclidean) statistical depth function in $\R^d$ in \citet{Zuo_Serfling2000}. Analogues of the additional conditions~\ref{cond: spherical usc} and~\ref{cond: spherical convexity} have been introduced in the analysis of the depth in $\R^d$ by \citet{Serfling2006}. The final condition~\ref{cond: nonrigidity} appears in \citet{Nagy_etal2023} for the first time. It is a minimal requirement on an angular depth function that guarantees that $\AD$ reflects the shape properties of the distribution $P \in \Prob[\Sph]$. 

Condition~\ref{cond: spherical maximality} operates with the notion of rotational symmetry \citep[Section~2.3.2]{Ley_Verdebout2017} of $P$. Recall that a distribution $P \in \Prob[\Sph]$ is said to be \emph{rotationally symmetric} around a direction $\mu \in \Sph$ if $X \sim P$ has the same distribution as $O X$ for any orthogonal matrix $O \in \R^{d \times d}$ that fixes $\mu$, that is $O \mu = \mu$. The center of rotational symmetry of $P$ is never unique; if $\mu \in \Sph$ is a center of rotational symmetry of $P$, then so is $-\mu \in \Sph$. Thus, the maximum on the left-hand side of~\eqref{eq: double maximality sphere} is necessary to be considered in~condition~\ref{cond: spherical maximality}.

The set of conditions~\ref{cond: rotational}--\ref{cond: nonrigidity} is not independent; clearly,~\ref{cond: spherical convexity} is a stronger version of~\ref{cond: spherical monotonicity}. We list both these requirements since in \citet{Nagy_etal2023}, it was argued that the quasi-concavity condition~\ref{cond: spherical convexity} takes a quite different meaning on the unit sphere than it does in the classical case of $\R^d$. In particular, in \citet[Theorem~3]{Nagy_etal2023} it is proved that~\ref{cond: spherical convexity} implies that $\AD$ must be constant on an open hemisphere in $\Sph$. Thus, for directional data, the requirement~\ref{cond: spherical convexity} of convexity of central regions is questionable, and the strictly weaker~\ref{cond: spherical monotonicity} may be preferable for some depths. We will, however, see that just like the (Euclidean) \tHD{} in $\R^d$, also \tAHD{} in $\Sph$ satisfies the stronger condition~\ref{cond: spherical convexity} with all its implications.

In the following subsections, we deal with conditions~\ref{cond: rotational}--\ref{cond: nonrigidity} one by one, but not necessarily in this order. We establish that each of our conditions is verified by \tAHD{}.

\subsection{Rotational invariance}  \label{section: rotational}

The validity of~\ref{cond: rotational} for $\AHD$ was proved already in \citet[Example~4.4.4]{Small1987}. It follows directly from the definition of $\AHD$; for completeness, we provide a proof.

\begin{theorem} \label{theorem: rotational}
The depth \tAHD{} satisfies condition~\ref{cond: rotational}.
\end{theorem}

\begin{proof}
Take $x \in \Sph$,  and $O\in \R^{d \times d}$ orthogonal. To find $\AHD(O x; P_{O X})$, we need to search through all halfspaces $H_{0,v}$ with $v \in \Sph$ such that $O x \in H_{0,v}$. Since $u \mapsto O u$ is a bijection of $\Sph$, we can equivalently write $v = O u$ and search over all $u \in \Sph$. The condition $O x \in H_{0, v} = H_{0, O u}$ then translates to
   \begin{equation} \label{eq: equivalent angles}
   0 \leq \ang{O x, O u} = (O x)\tr (O u) = x\tr O\tr O u = x\tr u = \ang{ x, u},    
   \end{equation}
i.e., it is equivalent with $x \in H_{0,u}$. Further, using~\eqref{eq: equivalent angles} again we can write 
    \[  
    P_{O X}(H_{0, O u}) = P\left( \left\{ y \in \Sph \colon O y \in H_{0, Ou} \right\} \right) = P\left( \left\{ y \in \Sph \colon y \in H_{0, u} \right\} \right) = P(H_{0,u}),
    \]
meaning that in both $\AHD(x;P)$ and $\AHD(O x; P_{O X})$ one considers the same collection of halfspaces. Necessarily,~\ref{cond: rotational} is true.
\end{proof}

The result of \citet[Example~4.4.4]{Small1987} is, actually, stronger than~\ref{cond: rotational}. It says that for any non-singular matrix $A \in \R^{d\times d}$ and
    \begin{equation}    \label{eq: general rotation}
    \varphi_A \colon \Sph \to \Sph \colon x \mapsto \frac{Ax}{\norm{ A x }}    
    \end{equation}
the invariance $\AHD(x;P) = \AHD(\varphi_A(x); P_{\varphi_A(X)})$ holds true for all $x \in \Sph$ and $X \sim P \in \Prob$. The map~\eqref{eq: general rotation} is a full-dimensional linear transform $x \mapsto A x$ in $\R^d$, followed by a projection back to $\Sph$. It is more general than the rotations considered in~\ref{cond: rotational}; for $O \in \R^{d \times d}$ orthogonal, we obtain $\varphi_O(x) = O x$ by~\eqref{eq: equivalent angles} and we recover~\ref{cond: rotational}. Unlike in~\ref{cond: rotational}, maps~\eqref{eq: general rotation} also allow ``stretching'' the sphere $\Sph$ to an ellipsoid $A \, \Sph = \left\{ A \, x \colon x \in \Sph \right\}$ before mapping it back to itself. 

\subsection{(Semi-)Continuity with consequences}

We now derive~\ref{cond: spherical usc} for \tAHD{}, but in doing so, we prove more: the depth $\AHD(x; P)$ is upper semi-continuous as a function of both arguments $x \in \Sph$ and $P \in \Prob$. For that result, we need to endow $\Prob$ with a topology; a natural one is the topology of weak convergence of measures.

\begin{theorem}	\label{theorem:angular depth continuity}
The angular halfspace depth \eqref{eq: angular halfspace depth} is upper semi-continuous as a function of $\left(x,P\right) \in \Sph \times \Prob[\Sph]$, meaning that
    \[	{\lim\sup}_{n \to \infty} \AHD\left(x_n; P_n\right) \leq \AHD\left(x;P\right)	\]
whenever $x_n \to x$ in $\Sph$, and $P_{n} \xrightarrow{w} P$ in $\Prob[\Sph]$ as $n \to \infty$. In particular, \tAHD{} satisfies condition~\ref{cond: spherical usc}. If, in addition, $P$ is smooth in the sense that 
    \begin{equation}	\label{smoothness of P} \tag{\textsf{S}}
    P\left( \partial H \right) = 0 \mbox{ for all }H \in \half_0,	
    \end{equation}
then \tAHD{} is continuous in both arguments, that is
    \[	\lim_{n \to \infty} \AHD\left(x_n; P_n \right) = \AHD\left(x;P\right)	\]
for $x_n \to x$ in $\Sph$, and $P_{n} \xrightarrow{w} P$ in $\Prob[\Sph]$ as $n \to \infty$. 
\end{theorem}

A simple consequence of the upper semi-continuity of the function \eqref{function f} is that the infimum in the definition of the depth \eqref{eq: angular halfspace depth} does not have to be attained. Indeed, take $d = 2$, and $P \in \Prob[{\Sph[1]}]$ given as a mixture of the uniform distribution on the half-circle $\left\{ x \in \Sph[1] \colon \ang{x, e_2} > 0 \right\} = \Sph[1]_+$ and an atom at $\left(-1,0\right)\tr$ with equal weights $1/2$. Consider the $\AHD(x;P)$ of $x = \left(1,0\right)\tr$. Evidently, $\AHD\left(x;P\right) = 0$ with $P\left( H_{0,u_n} \right) \to 0$ as $u_n = \left( \cos(-\pi/2+1/n), \sin(-\pi/2+1/n) \right)\tr \to (0,-1)\tr$, but no closed halfspace $H_{0,u}$ with $P\left(H_{0,u}\right) = 0$ exists. The problem with the non-existence of a minimizing halfspace $H_{0,u}$ satisfying $\AHD(x; P) = P(H_{0,u})$ can be resolved by considering so-called flag halfspaces \citep{Pokorny_etal2022}; we develop that theory in Section~\ref{section: flag} below.

\subsection{Quasi-concavity of level sets}

Just as for \tHD{}, also \tAHD{} has convex upper level sets. In the following theorem, we show a stronger claim: the level set    
    \[  \AHD_\alpha(P) = \left\{ x \in \Sph \colon \AHD(x; P) \geq \alpha \right\}  \] 
can be written as an intersection of specific open hemispheres in $\Sph$. Condition~\ref{cond: spherical convexity} for $\AHD$ follows immediately since an intersection of (spherical) convex sets is always (spherical) convex.

\begin{theorem} \label{theorem: quasi-concavity}
For any $P \in \Prob$ and $\alpha \geq 0$ we have
    \begin{equation}    \label{eq: level set}
    \AHD_\alpha(P) = \bigcap \left\{ G \colon G\compl\in \half_0 \mbox{ and }P(G\compl)<\alpha \right\}.    
    \end{equation}
In particular, \tAHD{} satisfies both~\ref{cond: spherical monotonicity} and~\ref{cond: spherical convexity}.
\end{theorem}

Comparing Theorem~\ref{theorem: quasi-concavity} with the related result for \tHD{} from \citet[Proposition~6]{Rousseeuw_Ruts1999}, we observe an intriguing discrepancy. While for \tHD{} the upper level sets $\HD_\alpha(P)$ can be written as intersections of \emph{closed} halfspaces whose complement has $P$-mass at most $\alpha$, in formula~\eqref{eq: level set} we used \emph{open} hemispheres. The following result shows that with closed hemispheres in~\eqref{eq: level set} we obtain only an inclusion.

\begin{theorem} \label{theorem: intersection}
For any $P \in \Prob$ and $\alpha \geq 0$ we have
    \begin{equation}    \label{eq: depth region as intersection of closed halfspaces}
    \bigcap\left\{H \colon H\in \half_0 \mbox{ and } P(H\compl)<\alpha \right\} \subseteq \AHD_\alpha(P).
    \end{equation}
\end{theorem}

It is interesting to see that the opposite inclusion from \eqref{eq: depth region as intersection of closed halfspaces} does not hold. An example can be found in Appendix~\ref{app: example}. Formula~\eqref{eq: level set} draws connections of \tAHD{} with \emph{spherical convex floating bodies} studied in convex geometry \citep{Besau_Werner2016}. Indeed, for $P \in \Sph$ uniform on a spherical convex body\footnote{A \emph{spherical convex body} is a closed spherically convex set $K \subseteq \Sph$ such that $\rad{K}$ has non-empty interior.} $K \subset \Sph$ the spherical convex floating body can be defined precisely as $\AHD_\alpha(P)$ with appropriate $\alpha > 0$ \citep[Definition~1]{Besau_Werner2016}. This observation parallels the connections between classical (Euclidean) floating bodies and \tHD{} leveraged in \citet{Nagy_etal2019}. A more detailed analysis of the relations of spherical floating bodies with \tAHD{} can be conducted using tools from~\citet[Section~3]{Laketa_Nagy2022}.

\subsection{Minimality at the anti-median and constancy on a hemisphere}    \label{section: flag}

The depth \tAHD{} has an interesting property observed first in \citet[Proposition~4.6]{Liu_Singh1992}. For any distribution on a sphere, there exists a hemisphere $H$ on which \tAHD{} is constant. For each $x \in H$, we then have that $\AHD(x; P)$ is equal to the minimum $P$-mass of a hemisphere in $\Sph$. Especially in connection with our property~\ref{cond: minimality}, it is important to note that this property does not hold true for closed hemispheres. 

\begin{example} \label{example: flag}
Take $P \in \Prob$ a mixture of a uniform distribution on $\Sph_+$ with weight $1/2$ and an atom of mass $1/2$ at some point $z\in \Sph_0$. Then, $\AHD(z; P)=1/2$ and $\AHD(x; P)=0$ for each $x \in \left(\Sph_- \cup \Sph_0\right)\setminus\{z\}$. For points $x \in \Sph_+$ we have $\AHD(x;P) \in (0,1/2)$. Thus, the unique angular halfspace median of $P$ is $z$. The infimum $P$-mass of a (closed) hemisphere is $0$, but because of the point mass at $z$, no closed hemisphere in $\Sph$ has constant null \tAHD{}. Still, $\AHD(-z; P) = 0 = \inf_{x \in \Sph} \AHD(x; P)$, and condition~\ref{cond: minimality} is satisfied for this particular $P$.
\end{example}

The appropriate context to study the set of minimum \tAHD{} is that of flag halfspaces, recently introduced in~\citet{Pokorny_etal2022}. There, a slightly more general version of the following definition can be found.

\begin{definition}
Define $\flag$ as the system of all sets $F$ of the form 
	\begin{equation}	\label{eq: flag halfspace}
	F = \{0\} \cup \left(\bigcup_{k=1}^d G_k\right).	
	\end{equation}
Here, 
    \begin{itemize}
    \item $G_d \subset \R^d$ is an open halfspace whose boundary passes through the origin. 
    \item For every $k = 1, \dots, d-1$, the set $G_k$ is an open halfspace inside the $k$-dimensional relative boundary of $G_{k+1}$, such that the relative boundary of $G_k$ passes through the origin.
    \end{itemize}
Any element of $\flag$ is called a \emph{flag halfspace}.
\end{definition}

A flag halfspace in $\R^3$ is the union of \begin{enumerate*}[label=(\roman*)] \item an open halfspace $G_3$ whose boundary passes through the origin, \item a relatively open halfplane $G_2$, inside the plane $\partial G_3$, whose relative boundary passes through the origin, and \item a ray $\{0\} \cup G_1$ from the origin $0 \in \R^3$ into one of the two directions in the line given by the relative boundary of $G_2$.\end{enumerate*}

Flag halfspaces are interesting because of their connections with \tHD{}. As proved in \citet[Theorem~1]{Pokorny_etal2022}, the infimum in the definition~\eqref{eq: halfspace depth first expression} of \tHD{} can always be replaced by a minimum, if one searches through flag halfspaces instead of closed halfspaces. An analogous result turns out to be true also for \tAHD{}, as stated in the following theorem.

\begin{theorem} \label{theorem: flag}
For any $P \in \Prob$ and $x \in \Sph$ we have
    \[  \AHD(x; P) = \min \left\{ P(F) \colon F \in \flag \mbox{ and } x \in F \right\}.  \]
In particular, there always exists $F \in \flag$ such that $x \in F$ and $\AHD(x; P) = P(F)$.
\end{theorem}

Armed with the notion of a flag halfspace, we now prove a sharp version of the claim on hemispheres of minimum \tAHD{}.

\begin{theorem} \label{theorem: minimizing}
Let $P \in \Prob$. Then there exists a flag halfspace $F \in \flag$ that satisfies 
    \begin{equation}    \label{eq: minimum flag}
    P(F) = \inf \left\{ P(G) \colon G \in \flag \right\}.
    \end{equation}
For every $x \in F$ we then have $\AHD(x; P) = \min_{y \in \Sph} \AHD(y; P)$. In particular, \ref{cond: minimality} is true for \tAHD{}.
\end{theorem}

As a consequence of Theorem~\ref{theorem: minimizing}, we obtain much more than just condition~\ref{cond: minimality} for \tAHD{}. It holds true that for any $F \in \flag$, we have $x \in F \cap \Sph$ if and only if $-x \notin F \cap \Sph$ \citep[Lemma~2.3]{Laketa_etal2022}. Thus, for any $x \in \Sph$, at least one of the antipodal directions $x, -x \in \Sph$ attains the minimum \tAHD{}
    \begin{equation}    \label{eq: min max}
    \min\{ \AHD(x; P), \AHD(-x; P) \} = \min_{y \in \Sph} \AHD(y; P).   
    \end{equation}
In our Example~\ref{example: flag}, for instance, we get that $F \in \flag$ satisfying \eqref{eq: minimum flag} is any flag halfspace of the form~\eqref{eq: flag halfspace} such that $G_d \cap \Sph = \Sph_-$ and $z \notin G_{d-1}$. This flag halfspace $F$ has null $P$-mass.

\subsection{Maximality at the center}

The following theorem states that condition~\ref{cond: spherical maximality} is satisfied for \tAHD{}.

\begin{theorem} \label{theorem: maximality}
Let $P \in \Sph$ be rotationally symmetric around $\mu \in \Sph$. Then 
    \[  \max\left\{ \AHD(\mu; P), \AHD(-\mu; P) \right\} = \max_{x \in \Sph} \AHD(x; P),    \]
and~\ref{cond: spherical maximality} is valid for \tAHD{}.
\end{theorem}

\subsection{Conclusion: Desirable properties of angular depths}

It remains to summarize our findings in Section~\ref{section: properties}: \tAHD{} verifies \ref{cond: rotational} by Theorem~\ref{theorem: rotational}; \ref{cond: spherical maximality} by Theorem~\ref{theorem: maximality}; \ref{cond: spherical monotonicity} by Theorem~\ref{theorem: quasi-concavity}; \ref{cond: minimality} by Theorem~\ref{theorem: minimizing}; \ref{cond: spherical usc} by Theorem~\ref{theorem:angular depth continuity}; \ref{cond: spherical convexity} by Theorem~\ref{theorem: quasi-concavity}; and \ref{cond: nonrigidity} because of Theorem~\ref{theorem: relation}. Overall, as argued in \citet{Nagy_etal2023}, it appears that \tAHD{} is the only angular depth function known in the literature that verifies all conditions \ref{cond: rotational}--\ref{cond: nonrigidity}. This, of course, does not mean that \tAHD{} is in any sense an optimal depth. It, however, hints that just as the classical \tHD{} in $\R^d$, also \tAHD{} in $\Sph$ has the potential to be useful in many applications in probability and statistics.

\section{Continuity and consistency properties} \label{section: additional properties}

We now focus on continuity and consistency properties of \tAHD{} that are finer in nature than the simple requirement~\ref{cond: spherical usc}. In Section~\ref{section: median}, we treat the set of \tAHD{}-based directional medians and show that this set is continuous as a set-valued mapping w.r.t. the topology of weak convergence in $\Prob$. Then, in Section~\ref{section: continuity}, we derive a uniform continuity theorem for \tAHD{} in the argument of measure. In Section~\ref{section: central regions}, we expand that theorem to the continuity of the central regions $\AHD_\alpha(P)$ from \eqref{eq: level set}. Finally, we summarize and apply all our previous advances to the sample \tAHD{} computed w.r.t. datasets in Section~\ref{section: large sample}, which gives remarkable strong uniform consistency properties for \tAHD{}. 

To state our results, recall that for compact sets $K, L \subset \R^d$ is the \emph{Hausdorff distance} of $K$ and $L$ defined as
    \begin{equation}    \label{eq: Haussdorf distance}
    \Hauss(K,L) = \max \left\{ \sup_{x \in K} \inf_{y \in L} \norm{ x - y }, \sup_{x \in L} \inf_{y \in K} \norm{ x - y }\right\}. 
    \end{equation}
For closed sets $K, L \subseteq \Sph$ we simply embed $\Sph$ into $\R^d$ canonically, and evaluate $\Hauss(K,L)$ in $\R^d$ with $\norm{\cdot}$ the Euclidean distance. It would, of course, be possible to modify~\eqref{eq: Haussdorf distance} to $\Sph$ by considering directly the arc distance length instead of the Euclidean distance. Thanks to the equivalence of all norms in finite-dimensional spaces, the topology of this modification remains the same as for $\Hauss$, and all our results thus hold true with both choices.

\subsection{Properties of the angular halfspace median} \label{section: median}

We are concerned with the continuity properties of the \tAHD{}-based set of directional medians, defined as the set of maximizers of \tAHD{} w.r.t. $P \in \Prob[\Sph]$
    \begin{equation}	\label{median mapping}
    \median(P) = \AHD_{\maxd(P)}(P) = \left\{ x \in \Sph \colon \AHD\left(x;P\right) = \maxd(P) \right\},	
    \end{equation}
with $\maxd(P) = \max_{y \in \Sph} \AHD\left(y;P\right)$. By Theorems~\ref{theorem:angular depth continuity} and~\ref{theorem: quasi-concavity}, we know that $\median(P)$ must be a non-empty compact (spherical) convex set. The following theorem extends results from \citet{Donoho_Gasko1992, Rousseeuw_Ruts1999}, and \citet{Mizera_Volauf2002} to the directional setup and \tAHD{}.

\begin{theorem} \label{theorem:angular depth median continuity}
The following properties hold true:
    \begin{enumerate}[label=(\roman*), ref=(\roman*)]
    \item \label{theorem i} The maximum depth mapping 
	\[	\maxd \colon \Prob[\Sph] \to \left[\frac{1}{d+1},1\right] \colon P \mapsto \sup_{y \in \Sph} \AHD\left(y;P\right)	\]    
    is upper semi-continuous. 
    \item \label{theorem ii} At any $P \in \Prob[\Sph]$ that satisfies \eqref{smoothness of P}, the mapping $\maxd$ is continuous and $\maxd(P) \leq 1/2$. 
    Further, the \tAHD{}-median mapping \eqref{median mapping} is at $P$ an outer semi-continuous set-valued mapping in the sense of \citet[Definition~5.4]{Rockafellar_Wets1998}. That is, for any $P_n \xrightarrow{w} P$ as $n \to \infty$ in $\Prob[\Sph]$ and $x_n \in \median(P_n)$ for each $n = 1, 2, \dots$, it holds true that all cluster points of the sequence $\left\{x_n\right\}_{n = 1}^\infty$ lie in $\median(P)$.
    \item \label{theorem v} Let $P \in \Prob[\Sph]$ satisfying~\eqref{smoothness of P} be such that $\median(P)$ is a singleton $\left\{ x \right\}$. Take any $P_n \xrightarrow{w} P$ as $n \to \infty$ in $\Prob[\Sph]$ and $x_n \in \median(P_n)$, $n = 1, 2, \dots$. Then there exists a sub-sequence $\left\{ x_{n(k)} \right\}_{k=1}^\infty$ of medians of $P_n$ such that $x_{n(k)} \to x$ as $k \to \infty$. In particular, the \tAHD{}-median mapping~\eqref{median mapping} is a continuous set-valued mapping in the sense of \citet[Definition~5.4]{Rockafellar_Wets1998}, and also continuous in the sense of convergence in the Hausdorff distance~\eqref{eq: Haussdorf distance}.
    \end{enumerate}
\end{theorem}

Note that in part~\ref{theorem i} of Theorem~\ref{theorem:angular depth median continuity}, we also claim that the depth \tAHD{} of a directional median on $\Sph$ cannot be lower than $1/(d+1)$. This bound is attained, for instance, for the peculiar atomic distribution $P\in\Prob[\Sph]$ described in the following example.

\begin{example} \label{ex:2}
Recall that $e_j \in \Sph$ is the $j$-th canonical vector, and write $e_{d+1} = \left(-1, \dots, -1\right)\tr/\sqrt{d} \in \Sph$. Consider $P\in\Prob[\Sph]$ the uniform measure supported in the set $\left\{e_1, \dots, e_d, e_{d+1}\right\}$. Then we have $\AHD\left(x;P\right) = 1/(d+1)$ for all $x \in \Sph$. The proof of this claim is in Appendix~\ref{app:ex2}.
\end{example}

The previous example is interesting because it demonstrates not only that there exists a distribution with maximum \tAHD{} equal to the lower bound from Theorem~\ref{theorem:angular depth median continuity}. Also, it shows that there is a distribution $X \sim P \in \Prob[\Sph]$ which fails to be origin-symmetric\footnote{A distribution $X \sim P\in\Prob[\Sph]$ is said to be \emph{origin symmetric} if $X$ and $-X$ have the same distribution.}, but its angular halfspace depth is constant on the whole sphere.

Parts~\ref{theorem ii} and~\ref{theorem v} of Theorem~\ref{theorem:angular depth median continuity} were stated for \tHD{} in \citet[Theorem~2]{Mizera_Volauf2002}. In the following example, we show that without the assumption of the uniqueness of the \tAHD{}-median of $P$, we cannot guarantee the continuity of the directional median set $\median(P)$. 

\begin{example}
Consider first $Q \in \Prob[\R]$ given as the mixture of uniform distributions on the intervals $[-2,-1]$ and $[1,2]$, each with weight $1/2$. The (standard Euclidean) halfspace median set of $Q$ is the whole interval $[-1,1]$. Now, for $\varepsilon > 0$ small and fixed, take $Q_n$ assigning mass $1/2 + \varepsilon/n$ to $[-2,-1]$ and mass $1/2 - \varepsilon/n$ to $[1,2]$. Certainly, $Q_n$ converges weakly to $Q$, but the median set of $Q_n$ is $\{ y_n \} = \left\{(n-4 \varepsilon)/(n+2 \varepsilon)\right\}$, which is contained inside the interval $[-2,-1]$. We see that the median mapping $\median$ for \tHD{} in $\R$ is outer semi-continuous but not continuous at $Q$. To obtain corresponding directional distributions, we project our setup to the upper semi-circle $\Sph[1]_+$ of the circle $\Sph[1]$ using the inverse gnomonic projection from Section~\ref{section: projection}. Then, we directly apply Theorem~\ref{theorem: relation}. 
\end{example}

\subsection{Continuity in measure}  \label{section: continuity}

Under an appropriate smoothness condition \eqref{smoothness of P}, the mapping \tAHD{} can be shown to be uniformly continuous w.r.t. the weak convergence of measures.

\begin{theorem} \label{theorem: uniform continuity}
Suppose $\left\{ P_n \right\}_{n=1}^\infty \subset \Prob[\Sph]$ is a sequence of measures such that $P_n \xrightarrow{w} P$ as $n \to \infty$, where $P \in \Prob[\Sph]$ satisfies the smoothness condition~\eqref{smoothness of P}. Then we can write
    \[  \sup_{x \in \Sph} \left\vert \AHD(x; P_n) - \AHD(x; P) \right\vert \xrightarrow[n\to\infty]{} 0. \]
\end{theorem}

Without the smoothness condition~\eqref{smoothness of P}, we cannot assert that the uniform continuity in Theorem~\ref{theorem: uniform continuity} is true. This can be seen in an example where $P$ is concentrated in $e_d \in \Sph$, but $P_n$ is uniform on a spherical cap in $\Sph$ around $e_d$ with (spherical) radius $1/n$. Then $P_n \xrightarrow[]{w} P$ as $n \to \infty$ and $\AHD(e_d; P) = 1$, but $\AHD(e_d; P_n) = 1/2$ for all $n$.

\subsection{Continuity of the central regions}  \label{section: central regions}

We now follow \citet{Dyckerhoff2017}, who proved that under certain conditions, the central regions $\HD_\alpha(P)$ from~\eqref{eq: central regions} are continuous in the Hausdorff distance as a function of $P \in \Prob[\R^d]$. We adapt those results from $\R^d$ to the setup of directional data in $\Sph$ and \tAHD{}. For that, we will need a modification of the strict monotonicity condition formulated in \citet{Dyckerhoff2017} for \tHD{} in $\R^d$. We phrase a related requirement for \tAHD{}; we say that \tAHD{} is \emph{strictly monotone} for $P \in \Prob[\Sph]$ if for each $\alpha \in \left(\min_{x \in \Sph} \AHD(x; P),\max_{x \in \Sph} \AHD(x;P)\right)$ we have
    \begin{equation}    \label{eq: strict monotonicity}
    \AHD_\alpha(P) = \cl{\left\{ x \in \R^d \colon \AHD(x; P) > \alpha  \right\}},   
    \end{equation}
where $\cl{A}$ stands for the closure of the set $A \subseteq \Sph$. Roughly speaking, strict monotonicity means that there are no regions of $\Sph$ of constant depth equal to $\alpha > \min_{x \in \Sph} \AHD(x; P)$. Of course, due to Theorem~\ref{theorem: minimizing}, we have to exclude the hemisphere of minimum $P$-mass (i.e., $\alpha > \min_{x \in \Sph} \AHD(x; P)$), since at $\alpha = \min_{x \in \Sph} \AHD(x; P)$ the condition~\eqref{eq: strict monotonicity} is never satisfied.

\begin{theorem}   \label{theorem: continuity of central regions}
Suppose that \tAHD{} is strictly monotone for $P \in \Prob[\Sph]$, and that $\left\{ P_n \right\}_{n=1}^\infty \subset \Prob[\Sph]$ is a sequence of measures that satisfies
    \begin{equation}    \label{eq: convergence condition}
    \lim_{n \to \infty} \sup_{x \in \Sph} \left\vert \AHD(x; P_n) - \AHD(x; P) \right\vert = 0.   
    \end{equation}
Then for any closed interval $A \subset (\min_{x \in \Sph} \AHD(x; P), \max_{x \in \Sph} \AHD(x; P))$ we can write
    \begin{equation}    \label{eq: uniform continuity of regions}
    \lim_{n \to \infty} \sup_{\alpha \in A} \Hauss\left( \AHD_\alpha(P_n), \AHD_\alpha(P) \right) = 0.
    \end{equation}   
\end{theorem}

The condition~\eqref{eq: convergence condition} is satisfied if $P$ is smooth (that is, \eqref{smoothness of P} is valid) by Theorem~\ref{theorem: uniform continuity}. The additional condition~\eqref{eq: strict monotonicity} of strict monotonicity of $P \in \Prob[\Sph]$ is more delicate. In the setup of \tHD{} in $\R^d$, it was argued in \citet[Section~4.3]{Laketa_Nagy2022} that smoothness of $P \in \Prob[\R^d]$ and the connectedness of its support already guarantee a variant of~\eqref{eq: strict monotonicity}. In $\Sph$, however, this is not enough, as we demonstrate in the following example.

\begin{example} \label{example: sets}
Take $X \sim P \in \Prob[{\Sph[1]}]$, where the random variable $X \in \Sph[1]$ is encoded by its angle $\theta \in (-\pi, \pi]$ with the positive first coordinate axis in $\R^2$. The measure $P$ is given as a mixture of four uniform distributions: \begin{enumerate*}[label=(\roman*)] \item of mass $1/2$ in the angle $(\pi/2, \pi]$, \item of mass $1/4$ in the angle $(0,\pi/4]$, \item of mass $1/8$ in the angle $[\pi/4, \pi/2)$, and \item of mass $1/8$ in the angle $[-3\pi/4, -\pi/2)$.\end{enumerate*} See also Figure~\ref{figure: strict monotonicity}. 

\begin{figure}
    \centering
    \includegraphics[width=.4\textwidth]{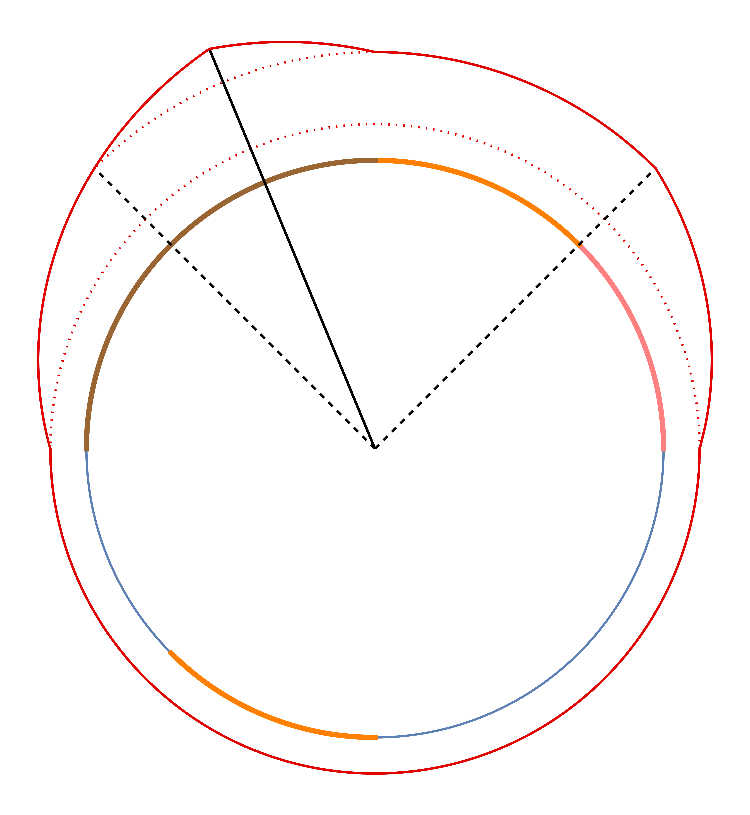} \quad
    \raisebox{1.5em}{\includegraphics[width=.5\textwidth]{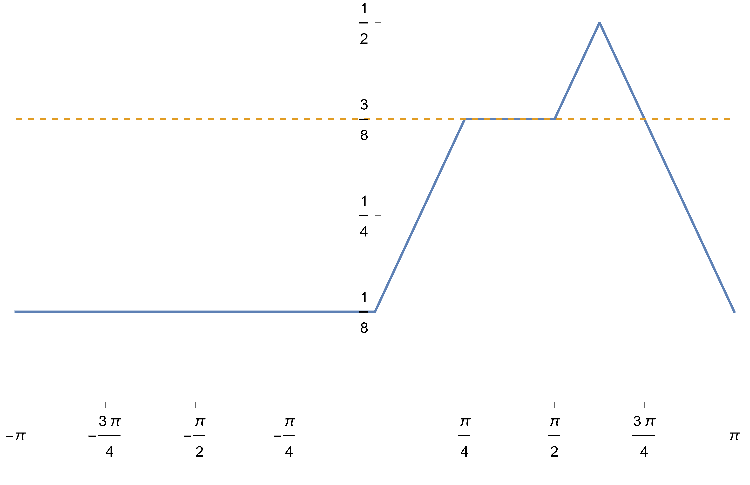}}
    \caption{Example~\ref{example: sets}: A distribution $P \in \Prob[{\Sph[1]}]$ that does not satisfy the condition of strict monotonicity~\eqref{eq: strict monotonicity} at $\alpha = 3/8$. In the left panel, the outer red curve corresponds to the \tAHD{} of $x(\theta) = (\cos(\theta), \sin(\theta) )\tr \in \Sph[1]$ as a function of the angle $\theta \in (-\pi,\pi]$; the depth \tAHD{} is maximized at the angle $\theta = 5\pi/8$ (solid straight line), and $\AHD_{3/8}(P)$ corresponds to the angles $\theta \in [\pi/4, 3 \pi/4]$ (wedge with dashed boundary lines). The depth \tAHD{} is constant for $\theta \in [\pi/4, \pi/2]$, meaning that \tAHD{} is not strictly monotone for $P$. In the right panel, we have the function $\theta \mapsto \AHD(x(\theta); P)$ (blue curve).}
    \label{figure: strict monotonicity}
\end{figure}

The minimum $P$-mass of a hemisphere is $1/8$, and $\AHD(x; P) = 1/8$ for all $x \in \Sph[1]_- \cup \Sph[1]_0$. For $x = (\cos(\theta), \sin(\theta))\tr \in \Sph[1]$ we get \tAHD{}
    \begin{equation}    \label{eq: computed ahd}
       \AHD(x; P) =    
        \begin{cases}
        1/8 & \mbox{for }\theta \in (-\pi,0], \\
        1/8 + \theta/\pi & \mbox{for }\theta \in (0,\pi/4], \\
        3/8 & \mbox{for }\theta \in (\pi/4, \pi/2], \\
        1/2 - \frac{\left\vert \theta - 5 \pi/8 \right\vert}{\pi} & \mbox{for }\theta \in (\pi/2, 3\pi/4], \\
        1/8 + \frac{\pi - \theta}{\pi} & \mbox{for }\theta \in (3\pi/4,\pi].
        \end{cases}    
    \end{equation}  
In particular, for $\alpha = 3/8$ we have that $\AHD_\alpha(P)$ corresponds to the interval of angles $\theta \in [\pi/4, 3\pi/4]$, but the (spherical) closure of the set $\left\{ x \in \Sph[1] \colon \AHD(x; P) > 3/8 \right\}$ is the interval of angles $\theta \in [\pi/2, 3\pi/4]$. The strict monotonicity condition~\eqref{eq: strict monotonicity} is, therefore, not satisfied for $\alpha = 3/8$. 

Construct now for $\varepsilon > 0$ small and fixed a sequence of measures $\left\{ P_n \right\}_{n=1}^\infty \subset \Prob[{\Sph[1]}]$ defined as similarly as $P$, but in the interval $\theta \in (\pi/2, \pi]$ we put $P_n$-mass $1/2 + (-1)^{n}\varepsilon$, and in the interval $(0,\pi/4]$ we assign $P_n$-mass $1/4 + (-1)^{n+1}\varepsilon$. In other words, for $n$ odd we increase the mass in $(0,\pi/4]$ slightly above $1/4$ and decrease the mass in $(\pi/2, \pi]$ below $1/2$, and for $n$ even the other way around. We certainly have that $P_n$ converges weakly to $P$ as $n \to \infty$, and since $P$ satisfies the smoothness condition~\eqref{smoothness of P}, Theorem~\ref{theorem: uniform continuity} also gives that $\sup_{x \in \Sph[1]} \left\vert \AHD(x; P_n) - \AHD(x; P) \right\vert$ vanishes as $n \to \infty$. Nevertheless, a simple computation as in~\eqref{eq: computed ahd} gives that for $\alpha = 3/8$ we have $\AHD_\alpha(P_n)$ contains the interval of angles $\theta \in (\pi/4, \pi/2]$ for $n$ odd, but not for $n$ even. As such, the sequence of sets $\left\{ \AHD_\alpha(P_n)\right\}_{n=1}^\infty$ does not converge in the Hausdorff distance to any set, and~\eqref{eq: uniform continuity of regions} cannot be true.
\end{example}

Note that the problem with Example~\ref{example: sets} does not rest in the fact that the support of $P$ is disconnected; one could easily add another mixture component supported on $\Sph[1]$ to $P$ with sufficiently low weight, and the same phenomenon appears. The core of the problem is in the symmetry of $P$ in the angle $\theta \in [\pi/4, \pi/2)$, giving that any halfspace $H_u \in \half_0$ with $u = (\cos(\theta), \sin(\theta))\tr$ and $\theta \in [3\pi/4, \pi/2]$ has the same $P$-mass equal to $\alpha = 3/8$. One natural way of enforcing~\eqref{eq: strict monotonicity} is to forbid these symmetries, as done in the following result.

\begin{theorem} \label{thm:xx}
Suppose that $P \in \Prob[\Sph]$ has a dominant hemisphere, meaning that there exists $u \in \Sph$ such that for all $B \subseteq \Sph \cap H_u$ of non-null spherical Lebesgue measure we have $P(B) > P(-B)$. Assume, in addition, that $P$ is smooth in the sense of~\eqref{smoothness of P}. Then is the condition~\eqref{eq: strict monotonicity} satisfied for $P$.
\end{theorem}

Another way of obtaining the strict monotonicity condition~\eqref{eq: strict monotonicity} is to use directly Theorem~\ref{theorem: relation} and assume that $P$ is supported in a hemisphere $H$. Then, we can assume that after gnomonic projection to the tangent hyperplane at the pole of $H$, the (non-negative) measure $P_\pm$ in $\R^{d-1}$ satisfies smoothness and contiguity of its support as assumed in \citet[Section~4.3]{Laketa_Nagy2022}.

\subsection{Large sample properties}    \label{section: large sample}

As our final result, we consider the problem of estimating \tAHD{} from a random sample. We consider a sequence $X_1, \dots, X_n \in \Sph$ of independent random variables defined on $\Omega$ sampled from the distribution $P \in \Prob$. We attach to each $X_i \equiv X_i(\omega)$ mass $1/n$ and denote the resulting (random) empirical measure by $P_n(\omega) \in \Prob[\Sph]$. When computing the depth, we typically estimate the depth of the (unknown) distribution $P$ by means of the sample depth, based on plugging the empirical measure $P_n(\omega)$ into the depth~\eqref{eq: angular halfspace depth} instead of $P$. The depth $\AHD(x; P_n(\omega))$ is called the sample angular halfspace depth of $x$. In the following result, we prove that the sample \tAHD{} almost surely uniformly approximates its population counterpart $\AHD(\cdot; P)$ as $n \to \infty$, and the same is true for the derived quantities $\median(P_n(\omega))$ and $\AHD_\alpha(P_n(\omega))$. 

\begin{theorem} \label{theorem:12}
Suppose that $X_1, X_2, \dots,$ is a sequence of independent random variables with distribution $P \in \Prob[\Sph]$, and let $P_n(\omega) \in \Prob[\Sph]$ be the (random) empirical measure corresponding to $X_1(\omega), \dots, X_n(\omega)$.
    \begin{enumerate}[label=(\roman*), ref=(\roman*)]
    \item \label{theorem sample i} Then we have
        \[  \PP\left( \left\{ \omega \in \Omega \colon \lim_{n \to \infty} \sup_{x \in \Sph} \left\vert \AHD(x; P_n(\omega)) - \AHD(x; P) \right\vert = 0 \right\} \right) = 1. \]
    \item If $P \in \Prob[\Sph]$ is such that \tAHD{} is strictly monotone for $P$, then for any closed interval $A \subset (\min_{x \in \Sph} \AHD(x; P), \max_{x \in \Sph} \AHD(x; P))$ we can write
        \[
        \PP\left( \left\{ \omega \in \Omega \colon \lim_{n \to \infty} \sup_{\alpha \in A} \Hauss\left( \AHD_\alpha(P_n(\omega)), \AHD_\alpha(P) \right) = 0 \right\} \right) = 1.  
        \]
    \item If $P \in \Prob[\Sph]$ satisfies the conditions of part~\ref{theorem v} of Theorem~\ref{theorem:angular depth median continuity}, then
        \[
        \PP\left( \left\{ \omega \in \Omega \colon \lim_{n \to \infty} \Hauss\left( \median(P_n(\omega)), \median(P) \right) = 0 \right\} \right) = 1.  
        \] 
    \end{enumerate}
\end{theorem}

\appendix 

\section{Proofs of the theoretical results}

\subsection{Proof of Theorem~\ref{theorem: relation}}

In the definition of \tAHD{} in~\eqref{eq: angular halfspace depth} we only consider halfspaces in $\R^d$ whose boundary passes through the origin. Denote by $G_u$ the intersection of $\Gplus$ and $H_{0,u} \in \half_0$. When considered as a subset of the affine space $\Gplus$, the set $G_u$ is a generalized halfspace: \begin{enumerate*}[label=(\roman*)] \item it is a halfspace in $\Gplus$ if $u \in \Sph\setminus\left\{ e_d, -e_d \right\}$, \item an empty set if $u = -e_d$, and \item equals $\Gplus$ if $u = e_d$.\end{enumerate*} By the definition~\eqref{eq: gnomonic projection} of the gnomonic projection $\xi$, a point $y \in \Sph\setminus\Sph_0$ lies in $H_{0,u}$ if and only if $\ang{y, u} = \ang{\xi(y), u} \ang{y, e_d} \geq 0$. If $y \in \Sph_+$, this is equivalent with $\xi(y) \in G_{u} = H_{0,u} \cap \Gplus$; for $y \in \Sph_-$ we have that $y \in H_{0,u}$ if and only if $\ang{\xi(y), u} \leq 0$, that is $\xi(y) \in G_{-u}$. In terms of $P_{\pm}$ we can thus express the $P$-mass of any $H_{0,u} \in \half_0$ as 
    \begin{equation}    \label{eq: pm halfspace}
    P\left(H_{0,u}\right) = P_+\left( G_u \right) + P_-\left( G_{-u} \right).    
    \end{equation}		
Writing $G_u^\circ$ and $\partial G_u$ for the relative interior and the relative boundary of the $(d-1)$-dimensional generalized halfspace $G_u$ in $\Gplus$, for any $u \in \Sph$ we have $P\left( \Sph_- \right) = P_-\left(\Gplus\right) = P_-\left( G_{-u}^\circ \right) + P_-\left( \partial G_{-u} \right) + P_-\left( G_{u}^\circ \right) = P_-\left( G_{-u} \right) + P_-\left( G_{u}^\circ \right)$, simply because the sets $G_{-u}$ and $G_u^\circ$ are disjoint and decompose $\Gplus$. Thus, we can rewrite~\eqref{eq: pm halfspace} to
    \begin{equation}    \label{eq: pm}	
    P\left(H_{0,u}\right) = P\left(\Sph_-\right) + P_+\left( G_u \right) - P_-\left( G_{u}^\circ \right).	
    \end{equation}
This holds true for any normal vector $u \in \Sph$, including $u = e_d$ or $-e_d$. 

Since we assumed in~\eqref{eq: zero equator} that $x \in \Sph_+$, we know that $x \in H_{0,u}$ if and only if $\xi(x) \in G_u$. Plugging~\eqref{eq: pm} into \eqref{eq: angular halfspace depth} we obtain
    \[
    \begin{aligned}
    \AHD(x;P) & = P\left(\Sph_-\right) + \inf \left\{ P_+\left( G_u \right) - P_-\left( G_{u}^\circ \right) \colon u \in \Sph \mbox{ and } x \in H_{0,u}\right\} \\
    & = P\left(\Sph_-\right) + \inf \left\{ P_+\left( G_u \right) - P_-\left( G_{u}^\circ \right) \colon u \in \Sph \mbox{ and } \xi(x) \in G_u \right\} \\
    & = P\left(\Sph_-\right) + \inf \left\{ P_+\left( H \right) - P_-\left( H^\circ \right) \colon H \in \genhalf \mbox{ and } \xi(x) \in H \right\}
    \end{aligned}
    \]
where in the last expression, it makes no difference whether the generalized halfspaces $\genhalf$ are considered in $\R^d$, or $\Gplus$, or $\R^{d-1}$ when identified with $\Gplus$. Considering $\genhalf$ in $\R^{d-1}$, we have proved the first part of our theorem. 

Suppose now that $P_-$ satisfies the continuity condition~\eqref{eq: negative smoothness}. Then, $P_-(\partial H) = 0$ for all $H \in \genhalf$ in $\R^{d-1}$ and thus $P_-(H^\circ) = P_-\left( H \right)$, which allows us to write $P_+\left( H \right) - P_-\left( H^\circ \right) = P_+(H) - P_-(H) = P_\pm(H)$ in~\eqref{eq: general AHD formula}.

\subsection{Proof of Theorem~\ref{theorem:angular depth continuity}}

We first use the portmanteau theorem \citep[Theorem~11.1.1]{Dudley2002} to show that the function
    \begin{equation}	\label{function f}
    f \colon \Prob[\Sph] \times \Sph \to [0,1] \colon \left(P,u\right) \mapsto P\left(H_{0,u}\right)
    \end{equation}
is upper semi-continuous at any measure $P\in\Prob[\Sph]$, and continuous at any $P$ that satisfies \eqref{smoothness of P}. To see that, fix $P \in \Prob[\Sph]$ and take $u_n \to u$ in $\Sph$. Since $u_n$ converges to $u$, it is possible to find a sequence of orthogonal matrices $\left\{ O_n \right\}_{n = 1}^\infty \subset \R^{d \times d}$ such that $O_n u_n = u$ for each $n = 1, 2, \dots$, and $O_n$ converges to the identity matrix $I_d \in \R^{d \times d}$. Define a measure $Q_n \in \Prob$ as the distribution of $O_n X_n$ with $X_n \sim P_n$. By Slutsky's theorem \citep[Theorem~2.13]{Jiang2010}, we then have that $Q_n \xrightarrow{w} P$ as $n \to \infty$. At the same time, for each $n = 1, 2, \dots$ we can write 
    \[ 
    O_n H_{0,u_n} = \left\{ O_n y \colon \ang{y, u_n} \geq 0 \right\} = \left\{ z \colon \ang{O_n\tr z, u_n} \geq 0 \right\} = H_{0, O_n u_n},
    \]
where we used $O_n^{-1} = O_n\tr$ as follows from the orthogonality of $O_n$ and therefore, we have
    \begin{equation}    \label{eq: equal halfspaces}   P_n(H_{0,u_n}) = \PP\left( X_n \in H_{0,u_n}\right) = \PP\left( O_n X_n \in O_n H_{0,u_n}\right) = Q_n(H_{0,u}). 
    \end{equation}
Thus, we have a fixed closed halfspace $H_{0,u}$ and a sequence of measures $\left\{Q_n\right\}_{n=1}^\infty \subset \Prob$ such that $Q_n \xrightarrow{w} P$ as $n \to \infty$. The portmanteau theorem \citep[Theorem~11.1.1]{Dudley2002} and~\eqref{eq: equal halfspaces} then give that
    \[  {\lim\sup}_{n \to \infty} f(P_n,u_n) = {\lim\sup}_{n \to \infty} P_n(H_{0,u_n}) = {\lim\sup}_{n \to \infty} Q_n(H_{0,u}) \leq P(H_{0,u}) = f(P,u) \]
as we wanted to show. In the special case when also~\eqref{smoothness of P} is valid, we obtain that $H_{0,u}$ is a continuity set of $P$, and the portmanteau theorem gives even 
    \[  {\lim}_{n \to \infty} f(P_n,u_n) = {\lim}_{n \to \infty} P_n(H_{0,u_n}) = {\lim}_{n \to \infty} Q_n(H_{0,u}) = P(H_{0,u}) = f(P,u). \]

Now for $x \in \Sph$ denote 
    \begin{equation}	\label{set v}
    v(x) = \left\{ u \in \Sph \colon x \in H_{0,u} \right\} = \left\{ u \in \Sph \colon \left\langle x, u \right\rangle \geq 0 \right\} = \Sph \cap H_{0,x}	
    \end{equation}
the set of all inner normals of halfspaces taken in the infimum in the definition~\eqref{eq: angular halfspace depth} of $\AHD(x; P)$. The set-valued mapping $v(x)$ is continuous in $x$ in the sense of Painlev\'e-Kuratowski convergence \citep[Chapter~5.B]{Rockafellar_Wets1998}, and \tAHD{} can be written as 
    \[	\AHD\left(x;P\right) = \inf_{u \in v(x)} f(P,u).	\]
We have verified all the assumptions of Berge's maximum theorem \citep[pp.~115--117]{Berge1997} on parametric optimization, which asserts that the depth function $\AHD \colon \Sph \times \Prob \to [0,1] \colon (x, P) \mapsto \AHD(x; P)$ is upper semi-continuous in both arguments for general $P$, and continuous in both arguments at any $P$ satisfying \eqref{smoothness of P}. 

\subsection{Proof of Theorem~\ref{theorem: quasi-concavity}}

Suppose first that $x \in \AHD_\alpha(P)$, meaning that $\AHD\left(x; P\right) \geq \alpha$. If $x$ does not belong to the right-hand side of \eqref{eq: level set}, then there must exist an open halfspace $G$ whose boundary passes through the origin, $P(G\compl) < \alpha$, and $x \notin G$. Take $H = G \compl \in \half_0$. Now, we have $x\in H \in \half_0$ and $P(H)<\alpha$, which contradicts $\AHD\left(x; P\right) \geq \alpha$. Necessarily, $\AHD_\alpha(P) \subseteq \bigcap \left\{ G \colon G\compl\in \half_0 \mbox{ and }P(G\compl)<\alpha \right\}$.

For the other direction, suppose that $x \notin \AHD_\alpha(P)$. Then we have $\alpha > \AHD(x;P) = \inf_{H \in \half_0 \colon x \in H} P(H)$, meaning that there exists a closed halfspace $H \in \half_0$ such that $x\in H$ and $P(H)<\alpha$. Consider the open halfspace $G = H\compl$. Then $x\notin G$, $G\compl \in \half_0$ and $P(G\compl) = P(H) < \alpha$, meaning that $x$ does not belong to the right-hand side of \eqref{eq: level set}. This gives $\AHD_\alpha(P) \supseteq \bigcap \left\{ G \colon G\compl\in \half_0 \mbox{ and }P(G\compl)<\alpha \right\}$, and proves our claim.

\subsection{Proof of Theorem~\ref{theorem: intersection}}

Take any $x \in \Sph$ that is not in $\AHD_\alpha(P)$. Then $\AHD\left(x; P\right) < \alpha$ and therefore, there is a closed halfspace $H\in \half_0$ such that $x\in H$ and $P(H) < \alpha$. First, we show that we can also assume that $H$ satisfies the additional property $x \in H^\circ$. Indeed, if $x \in \partial H$ and $u \in \Sph$ is the inner normal of $H$, we choose $u_n = \cos(1/n)u + \sin(1/n)x$ for $n = 1, 2, \dots$. Since $\ang{x,u} = 0$ we have
    \[  \norm{ u_n }^2 = (\cos(1/n))^2 \ang{u, u} + (\sin(1/n))^2 \ang{x, x} = 1, \]
and thus $u_n \in \Sph$. In addition, 
    \[  \ang{x, u_n} = \cos(1/n) \ang{x, u} + \sin(1/n) \ang{x, x} = \sin(1/n) > 0,  \]
and for $H_n = H_{0,u_n}$ we can write $x \in H_n^\circ$ for all $n$. Halfspaces $H_n \in \half_0$ correspond to slightly tilted $H$ so that $x$ is moved to the interior of $H_n$. In the proof of Theorem~\ref{theorem:angular depth continuity}, we proved that the function~\eqref{function f} given by $u \mapsto P(H_{0,u})$ is upper semi-continuous. Since $u_n \to u$, we thus have
    \[  \limsup_{n\rightarrow \infty} P(H_n) \leq P(H) < \alpha. \]
Necessarily, there must exist $n = 1, 2, \dots$ such that $P(H_n)<\alpha$. For such $n$ we have $x \in H_n^\circ$ as we wanted to show. 

Take now this $H\in \half_0$ with the properties $x\in H^\circ$ and $P(H) < \alpha$, and denote by $\tilde{H}=(H^\circ)\compl$ the complementary closed halfspace. Then $x \notin \tilde{H}$, but $\tilde{H} \in \half_0$ and $P(\tilde{H}\compl)=P(H^\circ) \leq P(H) <\alpha$. Thus, $x$ does not belong to the left-hand side of~\eqref{eq: depth region as intersection of closed halfspaces}.

\subsection{Example to Theorems~\ref{theorem: quasi-concavity} and~\ref{theorem: intersection}} \label{app: example}

%
It is convenient to define $P \in \Prob[{\Sph[2]}]$ by means of its gnomonic projection $P_\pm \in \MeasS[\R^2]$ from~\eqref{eq: gnomonic measure}. We then project $P_\pm$ to $\Sph[2]$ to obtain $P$, and use formula~\eqref{eq: general AHD formula} to evaluate \tAHD{}. For that we simplify the notation and write for $y \in \R^2$
    \begin{equation*}	
    \SHD\left(y;P_{\pm}\right) = \inf \left\{ P_{\pm}\left(H \right) + P_-\left(\partial H\right) \colon H \in \mathcal H^* \mbox{ and } y \in H \right\},	
    \end{equation*}
giving by~\eqref{eq: general AHD formula} that $\AHD(x;P) = P\left(\Sph_-\right) + \SHD\left(\xi(x);P_{\pm}\right)$.

We take the signed measure $P_\pm \in \MeasS[\R^2]$ with three atoms, of $P_\pm$-mass $1/5$ each, at points $y_1 = (-1,0)\tr$, $y_2 = (0,0)\tr$ and $y_3 = (1,0)\tr$, and two atoms, of mass $-1/5$ each, at points $z_1 = (0,-1)\tr$ and $z_2 = (0,1)\tr$. The total $P_\pm$-mass of $\R^2$ is $1/5$; any closed halfspace $H \subset \R^2$ that contains both $z_1$ and $z_2$ must also contain $y_2$. Thus, no closed halfspace in $\R^2$ can have $P_\pm$-mass smaller than $-1/5$, and $\SHD(y; P_\pm) \geq -1/5$ for all $y \in \R^2$. A halfspace $H$ with $P_\pm(H) = -1/5$ and $y \in H$ that does not contain any atoms of $P_\pm$ in $\partial H$ can be found for any $y \notin A$, where $A$ is the closed line segment between $y_1$ and $y_3$. Thus, $\SHD(y; P_\pm) = -1/5$ for all $y \notin A$. For any $y \in A$, we can always find a halfspace $H$ with $P_\pm(H) = 0$ and $y \in H$. We have that $\SHD(y; P_\pm) = 0$ for $y \in A$. 

\begin{figure}[htpb]
\includegraphics[width=.45\textwidth]{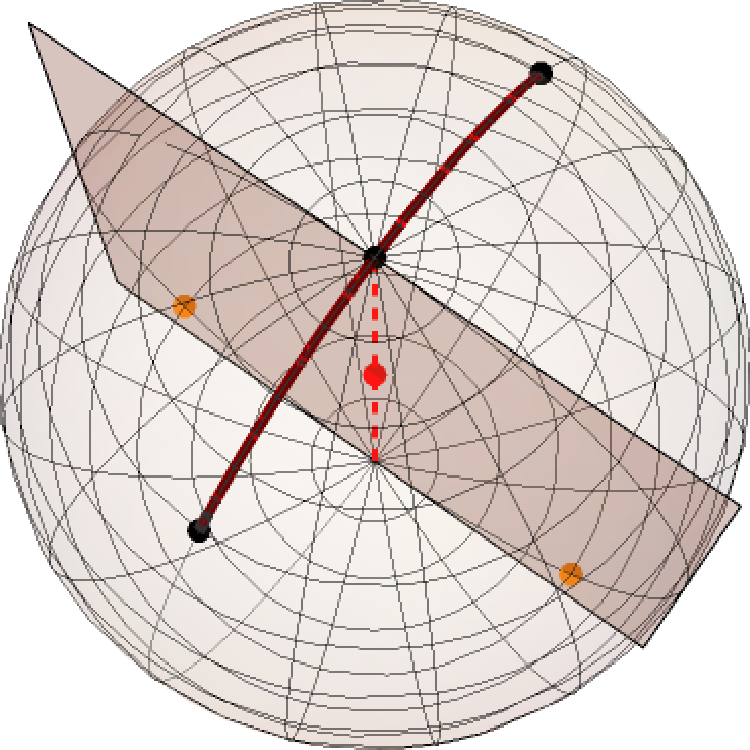} \quad \includegraphics[width=.45\textwidth]{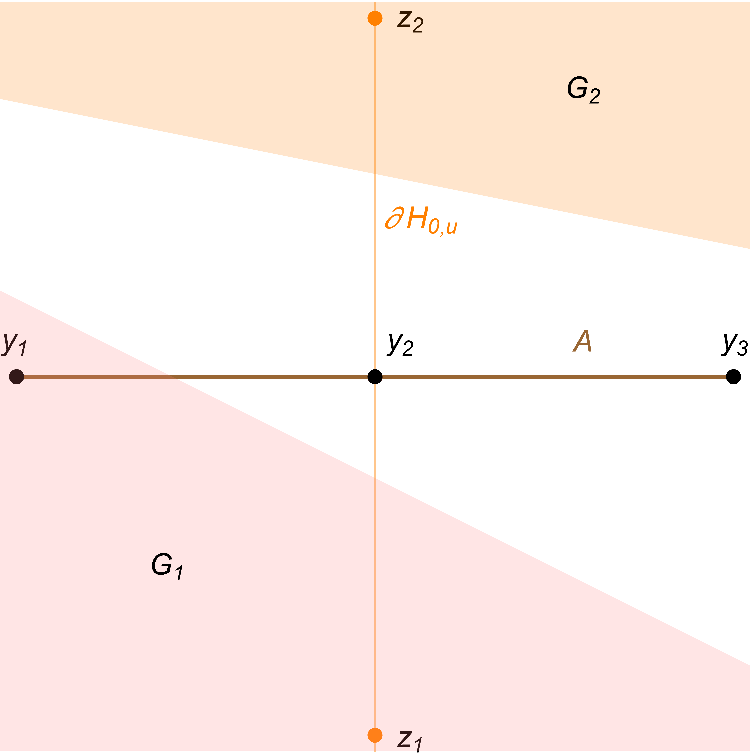}
\caption{The setup from Appendix~\ref{app: example}: The five points on the sphere $\Sph[2]$ (three black points in $\Sph[2]_+$ and two orange points in $\Sph[2]_-$; the origin is displayed in red. The depth \tAHD{} of the corresponding measure is equal to $1/5$ for all $x \notin B$ and $2/5$ for all $x \in B$, where $B$ is the arc displayed in brown. The plane on the left-hand side passing through the origin separates $\R^3$ into two open halfspaces, each containing only a single black point. In the right-hand panel, we see the same setup in the gnomonic projection in the plane $\R^2$. The pink halfplane $G_1$ contains $P_\pm$-mass $0$, while the orange one $G_2$ contains $P_\pm$-mass $-1/5$.}
\end{figure}

We now use the inverse gnomonic projection and formula~\eqref{eq: general AHD formula} to transfer our setup from $\R^2$ to $\Sph[2]$. The points $y_1$, $y_2$, $y_3$ map to $x_1 = (-1/\sqrt{2}, 0, 1/\sqrt{2})\tr$, $x_2 = (0,0,1)\tr$, $x_3 = (1/\sqrt{2}, 0, 1/\sqrt{2})\tr$ in $\Sph[2]_+$, while $z_1$ and $z_2$ map to $x_4 = (0, -1/\sqrt{2}, -1/\sqrt{2})\tr$ and $(0, 1/\sqrt{2}, -1/\sqrt{2})\tr$, respectively. Our atomic measure $P \in \Prob[{\Sph[2]}]$ has five atoms $x_i$, $i = 1, \dots, 5$, each of $P$-mass $1/5$. Using Theorem~\ref{theorem: relation} we get that $\AHD(x;P) = 1/5$ for all $x \in \Sph[2] \setminus B$ and $\AHD(x; P) = 2/5$ for $x \in B$, where $B = \left\{ \cos(\alpha) x_1 + \sin(\alpha) x_3 \colon \alpha \in [0,\pi/2] \right\}$ is the shorter arc between $x_1$ and $x_3$. We obtain $\AHD_{2/5}(P)=B$. Take now a closed halfspace $H = H_{0,u} \in \half_0$ with $u = (1,0,0)\tr$. We have $P(H\compl) = P(\{x_1\}) = 1/5 < 2/5$, but $H \cap \AHD_{2/5}(P)$ is one half of the arc $B$. Thus, for $\alpha = 2/5$ and, say, $x_1 \in \AHD_{\alpha}(P)$ we have $H \in \half_0$ that satisfies $P(H\compl)<\alpha$, but $x_1 \notin H$. Therefore, the inequality in~\eqref{eq: depth region as intersection of closed halfspaces} is strict.

\subsection{Proof of Theorem~\ref{theorem: flag}}

The proof is a direct adaptation of the proof of \citet[Theorem~1]{Pokorny_etal2022}.    

\subsection{Proof of Theorem~\ref{theorem: minimizing}}

The first part of the statement follows directly from \citet[Theorem~1]{Pokorny_etal2022}. Indeed, the right-hand side of~\eqref{eq: minimum flag} is by that theorem precisely the (Euclidean) halfspace depth of the origin $0 \in \R^d$ w.r.t. the measure $P$ when considered in the ambient space $\R^d$. Theorem~1 in \citet{Pokorny_etal2022} then states that $\HD(0; P)$ can also be expressed in terms of our flag halfspaces~\eqref{eq: flag halfspace}, and a flag halfspace satisfying~\eqref{eq: minimum flag} always exists. 

For the second part of our claim, note that since by Theorem~\ref{theorem: flag} also \tAHD{} can be expressed using flag halfspaces, for every $x \in F$ from~\eqref{eq: minimum flag}, we must have $\AHD(x; P) = P(F)$, directly because of~\eqref{eq: minimum flag} and the definition~\eqref{eq: angular halfspace depth} of \tAHD{}.

\subsection{Proof of Theorem~\ref{theorem: maximality}}

For any orthogonal matrix $O \in \R^{d \times d}$ such that $O \mu = \mu$ we know, by the assumption of rotational symmetry of $P$, that $P_{OX} = P$. Thanks to the rotational invariance~\ref{cond: rotational} of \tAHD{} from Theorem~\ref{theorem: rotational} we get
    \begin{equation}    \label{eq: constancy}
    \AHD(x; P) = \AHD(Ox; P_{OX}) = \AHD(Ox; P). 
    \end{equation}
Thus, \tAHD{} of $x \in \Sph$ must be constant on spheres that are cut from $\Sph$ by hyperplanes orthogonal to $\mu$. In other words, $\AHD(x; P)$ depends only on $\ang{x, \mu}$. 

By Theorem~\ref{theorem: minimizing}, we know that (at least) one of the points $\mu, -\mu \in \Sph$ lies in the region of minimum \tAHD{} of $P$. Without loss of generality, let $-\mu$ be this point, meaning that
    \begin{equation}    \label{eq: minimum at -mu}
    \AHD(-\mu; P) = \min_{x \in \Sph} \AHD(x; P). 
    \end{equation}
We will show that $\AHD(\mu;P) = \max_{x \in \Sph} \AHD(x;P)$. Suppose for contradiction that there exists $y \in \Sph \setminus\left\{ \mu, -\mu \right\}$ such that 
    \begin{equation} \label{eq: contradiction}
    \alpha = \AHD(y;P) > \AHD(\mu;P) \geq \AHD(-\mu;P),
    \end{equation}
the second inequality following trivially from~\eqref{eq: minimum at -mu}. In that case, formula \eqref{eq: constancy} gives that $\AHD(x;P) = \alpha$ for all points $x$ on the set $S = \left\{ x \in \Sph \colon \ang{x, \mu} = \ang{y, \mu} \right\}$. Necessarily, $\AHD_\alpha(P) \supseteq S$ but $\mu, -\mu \notin \AHD_\alpha(P)$. By Theorem~\ref{theorem: quasi-concavity}, the set $\AHD_\alpha(P)$ must be spherically convex, meaning that the convex hull of $\rad{S}$ intersected with $\Sph$ is a subset of $\AHD_\alpha(P)$. Now, if $\ang{y, \mu} \ne 0$, the convex hull of $\rad{S}$ necessarily contains either $\mu$ or $-\mu$, which would contradict~\eqref{eq: contradiction}. The only remaining case is that $\ang{y, \mu} = 0$. That is, however, also impossible thanks to our Theorem~\ref{theorem: minimizing} that entails~\eqref{eq: min max}. Indeed, if for some $y \in S$ we had~\eqref{eq: contradiction}, then~\eqref{eq: min max} gives that $\AHD(-y;P) = \AHD(-\mu;P) < \alpha$, but at the same time $-y \in S$ gives $\AHD(-y;P) = \alpha$, a contradiction. We have found that one of the points $\mu$ or $-\mu$ must maximize the angular halfspace depth of $P$ as we wanted to show.

\subsection{Proof of Theorem~\ref{theorem:angular depth median continuity}}

The following auxiliary lemma will be useful. 

\begin{lemma}	\label{lemma:Donoho}
Let $U \subset \Sph$ be a set that is either open or finite in the sphere $\Sph$. The following are equivalent:
	\begin{enumerate}[label=(\Alph*), ref=(\Alph*)]
	\item \label{Donoho A} the origin lies in the convex hull of $U$;
	\item \label{Donoho B} the origin lies in the convex hull of $R = \rad{U}$;
	\item \label{Donoho C} $\bigcup_{r \in R} H_{0,r} = \Sph$;
	\item \label{Donoho D} $\bigcup_{u \in U} H_{0,u} = \Sph$.
	\end{enumerate}
\end{lemma}

\begin{proof}[Proof of Lemma~\ref{lemma:Donoho}]
Since $U \subset R$, the convex hull of $U$ is a subset of the convex hull of $R$, and \ref{Donoho A} implies \ref{Donoho B}. For the opposite implication, suppose that the origin lies in the convex hull of $R$. Carath\'eodory's theorem \citep[Theorem~1.1.4]{Schneider2014} gives that there exists $d+1$ (not necessarily distinct) points $r_1, \dots, r_{d+1} \in R$ such that 
	\begin{equation}	\label{Caratheodory}
	0 = \sum_{i=1}^{d+1} \lambda_i \, r_i \mbox{ for some $0 \leq \lambda_i$ with $\sum_{i=1}^{d+1} \lambda_i = 1$.}
	\end{equation}
Since $R = \rad{U}$, being a radial extension, never contains the origin $0 \in \R^d$, we have that $\sum_{i=1}^{d+1} \lambda_i \norm{ r_i } > 0$. From \eqref{Caratheodory} we see that
	\[	\sum_{i=1}^{d+1} \frac{\lambda_i \norm{ r_i }}{\sum_{j=1}^{d+1} \lambda_j \norm{ r_j }} \frac{r_i}{\norm{ r_i }} = 0,	\]
and setting $u_i = r_i/\norm{ r_i } \in U$ and $\gamma_i = \lambda_i \norm{ r_i }/\left(\sum_{j=1}^{d+1} \lambda_j \norm{ r_j } \right)$ yields that the origin is also contained in the convex hull of $U$. Therefore \ref{Donoho B} implies \ref{Donoho A}.

Statements \ref{Donoho C} and \ref{Donoho D} are clearly equivalent since $H_{0,r} = H_{0,u}$ for each $r \in R$ and $u = r/\norm{ r }$. 

Note that if $U$ is supposed to be open in the sphere, then the set $R$ from \ref{Donoho B} is open in $\R^d$. By \citet[Theorem~1.1.10]{Schneider2014}, we then know that the convex hull $C$ of $R$ is open in $\R^d$. Suppose first that \ref{Donoho B} is not true. The Hahn-Banach theorem \citep[Theorem~1.3.4]{Schneider2014} gives that there exists a vector $x \in \Sph$ such that $H_{0,x} \cap C = \emptyset$, or equivalently $\left\langle x, r \right\rangle < 0$ for all $r \in C$. The strict inequality follows either from the fact that $C$ is open (in the case of $U$ open), or from the finiteness of $U$. The inequality $\left\langle x, r \right\rangle < 0$ however translates into $x \notin H_{0,r}$ for all $r \in R$, which gives that also \ref{Donoho C} cannot be true. 

Suppose finally that \ref{Donoho C} is violated. Then there exists $x \in \Sph$ that is not contained in any $H_{0,r}$, or equivalently $\left\langle x, r \right\rangle < 0$ for all $r \in R$. If the origin was contained in the convex hull $C$ of $R$, \citet[Theorem~1.1.4]{Schneider2014} again gives that we could write \eqref{Caratheodory} for some $r_1, \dots, r_{d+1} \in R$. The inner product of \eqref{Caratheodory} with $x$ then gives $0 = \sum_{i=1}^{d+1} \lambda_i \, \left\langle x, r_i \right\rangle < 0$, a contradiction. Thus, \ref{Donoho B} implies \ref{Donoho C}, and the lemma is proved.
\end{proof}

\textbf{Part~\ref{theorem i}.} We begin by showing that the depth $\maxd(P)$ of an \tAHD{}-median cannot be less than $1/(d+1)$. Suppose for contradiction that for some $P \in \Prob[\Sph]$ we have $\maxd(P) < c < 1/(d+1)$ for $c > 0$. Then, for any $x \in \Sph$ there exists a halfspace $H_{0,u(x)} \in \half_0$ that contains $x$, and $P(H_{0,u(x)}) < c$. In other words, there is a covering of $\Sph$ by halfspaces whose $P$-probability is less than $c$. Consider the collection of all inner normals $U = \left\{ u(x) \right\}_{x \in \Sph}$ of such halfspaces. By the upper semi-continuity of function $f$ from \eqref{function f}, the lower level set $U$ of $f$ is open in $\Sph$ \citep[Theorem~1.6]{Rockafellar_Wets1998}. The auxiliary Lemma~\ref{lemma:Donoho} below then gives that the origin must be contained in the convex hull of $U$. Carath\'eodory's theorem \citep[Theorem~1.1.4]{Schneider2014}, in turn, gives that in that case, the origin must be contained in a convex hull of at most $d+1$ (not necessarily distinct) elements of the set $U$ denoted by $u_1, \dots, u_{d+1} \in \Sph$. Using once again Lemma~\ref{lemma:Donoho}, we have that these $d+1$ elements induce a covering of $\R^d$, that is $\bigcup_{i=1}^{d+1} H_{0,u_i} = \R^d$ with $P\left(H_{0,u_i}\right) < c$ for all $i = 1,\dots, d+1$. That means
    \[	1 = P\left(\R^d\right) = P\left( \bigcup_{i=1}^{d+1} H_{0,u_i} \right) \leq \sum_{i=1}^{d+1} P\left(H_{0,u_i}\right) \leq (d+1) c < 1,	\]
a contradiction.

For the upper semi-continuity of the maximum depth mapping $\maxd$, Theorem~\ref{theorem:angular depth continuity} gives that the function $\AHD\left(\cdot;\cdot\right)$ is upper semi-continuous in both arguments. Berge's maximum theorem \citep[Theorem~2 on p.~116]{Berge1997} then directly yields the conclusion.

\textbf{Part~\ref{theorem ii}.} Under the smoothness assumption \eqref{smoothness of P}, Theorem~\ref{theorem:angular depth continuity} gives that $\AHD\left(\cdot;\cdot\right)$ is continuous at $P$. \citet[Maximum theorem on p.~117]{Berge1997} then entails the continuity of $\maxd$, and the outer semi-continuity of $\median$.

Now, let $x \in \median(P)$. Take any hyperplane $G$ containing $x$ and the origin in $\R^d$, and consider the two halfspaces $H, H' \in \half_0$ determined by $G$. Because $P$ is smooth, we have $P(G)=0$ and consequently $P(H)+P(H')=1$, which gives $\maxd(P) = \AHD(x; P) \leq \min\left\{ P(H), 1 - P(H) \right\} \leq 1/2$.

\textbf{Part~\ref{theorem v}.} By Theorem~\ref{theorem:angular depth continuity} we know that $x_n \in \median(P_n)$ always exists. Since $\Sph$ is a compact set, there exists a sub-sequence $\left\{ x_{n(k)} \right\}_{k = 1}^\infty$ of that sequence of directional medians such that $x_{n(k)}$ converges in $\Sph$ as $k \to \infty$. Thanks to part~\ref{theorem ii} of this theorem, we know that $x_{n(k)}$ then must converge to $x$, the unique median of $P$. We just proved the inner semi-continuity of the median mapping~\eqref{median mapping}, which together with its outer semi-continuity from part~\ref{theorem ii} of this theorem gives the continuity of~\eqref{median mapping} in the sense of \citet[Definition~5.4]{Rockafellar_Wets1998}. For a sequence of closed sets $\left\{ K_n \right\}_{n=1}^\infty$ that are all contained in a bounded subset of $\R^d$, the convergence in the Hausdorff distance $\lim_{n\to\infty} \Hauss\left(K_n, K\right) = 0$ is equivalent with the convergence of sets induced by the notion of continuity of set-valued mappings discussed above \citep[Chapter 4C and p.~144]{Rockafellar_Wets1998}. 

\subsection{Proof of Example~\ref{ex:2}}    \label{app:ex2}

First note that since $0 = \sum_{i=1}^{d+1} e_i/(d+1)$, Lemma~\ref{lemma:Donoho} gives that the collection of halfspaces $H_{0,e_i}$, $i=1,\dots,d+1$, covers $\Sph$. In particular, for any $x \in \Sph$ there must exist $i=1,\dots,d+1$ such that $\ang{ x, e_i } \geq 0$, implying $e_i \in v(x)$ for $v(x) = \left\{ u \in \Sph \colon x \in H_{0,u} \right\} = \Sph \cap H_{0,x}$, see also \eqref{set v}. Thus, any halfspace $H_{0,u} \in \half_0$ with $u \in v(x)$ contains at least one atom of $P$, and $\AHD\left(x;P\right) \geq 1/(d+1)$. We now construct a halfspace in $\half_0$ that contains $x$ and exactly one atom of $P$. Since $x = \left(x_1, \dots, x_d\right)\tr \in \Sph$, two situations are possible: \begin{enumerate*}[label=(\roman*)] \item either all coordinates of $x$ are non-positive, or \item for some $j = 1,\dots,d$ the coordinate $x_j$ is strictly positive. \end{enumerate*} 

In the first case, we consider the halfspace $H_{0,e_{d+1}}$. Then, $\ang{ x, e_{d+1} } > 0$, i.e. $x \in H_{0,e_{d+1}}$, and at the same time $\ang{ e_i, e_{d+1} } = -1$ for any $i \ne d+1$, meaning that $e_i \notin H_{0,e_{d+1}}$. We get that $H_{0,e_{d+1}}$ contains $x$ and $e_{d+1}$, but no other atom of $P$. 

In the second case, let $j = 1, \dots, d$ be the first index with $x_j > 0$. Take $\varepsilon > 0$ small enough, and consider $u = e_j - \left(\varepsilon, \dots, \varepsilon\right)\tr$ and the halfspace $H_{0,u/\norm{ u }}$. A simple calculation gives
    \[
    \begin{aligned}
    \ang{ e_j, u } & = 1 - \varepsilon, &  
    \ang{ e_i, u } & = - \varepsilon, \mbox{ for $i=1,\dots,d$, $i\ne j$}, \\
    \ang{ x, u } & = x_j (1 - \varepsilon) - \varepsilon \sum_{i \ne j} x_i,  & \mbox{ and } \quad
    \sqrt{d} \ang{ e_{d+1}, u } & = \varepsilon (d-1) - (1-\varepsilon), \\
    \end{aligned}
    \]
meaning that for $\varepsilon > 0$ small enough, $x, e_j \in H_{0,u/\norm{ u }}$, but $e_i \notin H_{0,u/\norm{ u }}$ for all $i=1,\dots,d+1$, $i \ne j$. Thus, $H_{0,u/\norm{ u }}$ contains only $x$ and a single atom of $P$, meaning that $\AHD\left(x;P\right) = 1/(d+1)$.   

\subsection{Proof of Theorem~\ref{theorem: uniform continuity}}

We have
    \begin{equation}    \label{eq: halfspace bound}
    \begin{aligned}
    \sup_{x \in \Sph} & \left\vert \AHD(x; P_n) - \AHD(x; P) \right\vert = \sup_{x \in \Sph} \left\vert \inf_{H \in \half_0 \colon x \in H} P_n(H) - \inf_{H \in \half_0 \colon x \in H} P(H) \right\vert \\
    & \leq \sup_{x \in \Sph} \sup_{H \in \half_0 \colon x \in H} \left\vert P_n(H) - P(H) \right\vert \leq \sup_{H \in \half_0} \left\vert P_n(H) - P(H) \right\vert.
    \end{aligned}
    \end{equation}
The final expression vanishes as $n \to \infty$ by, e.g., \citet[Theorem~A.3]{Nagy_etal2016}.

\subsection{Proof of Theorem~\ref{theorem: continuity of central regions}}

We apply the proof of implication \textbf{(ComD)} $\Rightarrow$ \textbf{(ComR)} from \citet[Theorem~4.5]{Dyckerhoff2017}. In that paper, only depths in $\R^d$ are considered. Nevertheless, the fact that one works in $\R^d$ is used only in several arguments in the proofs, and many of the derivations work precisely in the same way for measures in $\Sph$. 

For starters, in \citet[page~4]{Dyckerhoff2017} it is assumed that the function $D$ in $\R^d$ must be a depth in the sense of Definition~2.1 in \citet{Dyckerhoff2017}. That requires that the upper level sets $D_\alpha(P) = \left\{ x \in \R^d \colon D(x;P) \geq \alpha \right\}$ are all \begin{enumerate*}[label=\textbf{(R\arabic*)}] \item affine equivariant (for $\alpha \geq 0$), \item bounded (for $\alpha > 0$), \item closed (for $\alpha > 0$), and \item star-shaped (for all $\alpha \geq 0$). \end{enumerate*} The requirement \textbf{(R1)} of affine equivariance is, however, not used anywhere in the proof of Theorem~4.5 in \citet{Dyckerhoff2017}. The (spherical counterparts of the) additional conditions \textbf{R2}--\textbf{R4} are all satisfied for \tAHD{}, because \begin{itemize*} \item $\AHD_\alpha(P) \subseteq \Sph$ for all $\alpha \geq 0$ and $\Sph$ is itself bounded, giving \textbf{R2}, \item each $\AHD_\alpha(P)$ is closed by Theorem~\ref{theorem:angular depth continuity}, which verifies \textbf{R3}, and \item (spherical) star-shapedness of $\AHD_\alpha(P)$ in \textbf{R4} follows immediately from the (spherical) convexity of $\AHD_\alpha(P)$ as proved in Theorem~\ref{theorem: quasi-concavity}. \end{itemize*} Thus, the advances proved in \citet{Dyckerhoff2017} can be used in our setup.

Mimicking the proof of Theorem~4.5 in \citet{Dyckerhoff2017} for a depth $D$ in $\R^d$, we see that the core of the proof actually lies in proving \begin{itemize*} \item the continuity of the function $\alpha \mapsto D_\alpha(P)$ in the Hausdorff metric~\eqref{eq: Haussdorf distance}, and \item the bound\end{itemize*}
    \begin{equation}    \label{eq: Theorem 4.4}
    D_{\beta + \gamma}(P) \subseteq D_\beta(P_n) \subseteq D_{\beta - \gamma}(P)   \quad \mbox{for all $n \geq N$ and $\beta \geq \alpha_1$},
    \end{equation}
where $\alpha_1$ is the lower endpoint of the interval $A \subset (0,\max_{x \in \R^d} D(x; P))$, $\gamma > 0$ is a small enough constant, and $N$ is a large enough integer. 

The continuity of the mapping $\alpha \mapsto D_\alpha(P)$ is obtained in \citet[Theorem~3.2]{Dyckerhoff2017}. There, it is asserted that strict monotonicity of the depth $D$ for $P \in \Prob[\R^d]$ at all $\alpha \in (0, \max_{x \in \R^d} D(x; P))$ is enough to have continuity of $\alpha \mapsto D_\alpha(P)$ in $\alpha \in (0,\max_{x \in \R^d} D(x; P)]$. The proof of Theorem~3.2 in \citet{Dyckerhoff2017} does not use the fact that we work in $\R^d$; it works equally well in $\Sph$. In our setup, however, we assume strict monotonicity in~\eqref{eq: strict monotonicity} only for $\alpha \in (\min_{x \in \Sph} \AHD(x; P),\max_{x \in \Sph} \AHD(x; P))$, which by the same argument gives continuity of $\alpha \mapsto \AHD_\alpha(P)$ in the Hausdorff distance~\eqref{eq: Haussdorf distance} for $\alpha \in (\min_{x \in \Sph} \AHD(x; P),\max_{x \in \Sph} \AHD(x; P)]$.

As the second ingredient, we need to prove~\eqref{eq: Theorem 4.4} for $D = \AHD$. That is implication $(ii) \Rightarrow (iii)$ in \citet[Theorem~4.4]{Dyckerhoff2017}. As before, the proof of this implication does not use the fact that the depth $D$ is defined in $\R^d$ and also works for $\AHD$ in $\Sph$. In fact, the proof can be simplified by dropping the argument of a bounded set $M$ since, in our setup, we could take $M = \Sph$ directly. In view of our restriction to $\alpha > \min_{x \in \Sph} \AHD(x; P)$ in proving continuity of $\AHD_\alpha(P)$, we obtain \eqref{eq: Theorem 4.4} for all $\beta \geq \alpha_1 > \min_{x \in \Sph} \AHD(x; P)$. 

Putting the two ingredients together, in the proof of implication \textbf{(ComD)} $\Rightarrow$ \textbf{(ComR)} from \citet[Theorem~4.5]{Dyckerhoff2017} we obtain~\eqref{eq: uniform continuity of regions} as needed.

Observe that in \citet{Dyckerhoff2017}, two additional conditions are assumed for the depth $D$ throughout the paper:
    \begin{itemize}
    \item The range condition \textbf{(RC)} from page 9 of \citet{Dyckerhoff2017}.
    \item The ``general assumption'' stating that for each $\alpha < \max_{x \in \Sph} \AHD(x; P)$ must the set $\AHD_\alpha(P)$ have non-empty interior. 
    \end{itemize}
None of these is necessary for the proof of~\eqref{eq: uniform continuity of regions}. 

\subsection{Proof of Theorem~\ref{thm:xx}}

The condition of existence of a dominant hemisphere immediately implies that $\Sph \setminus H_{u}$ is an open hemisphere of minimum depth as in Theorem~\ref{theorem: flag}. The smoothness of $P$ and the continuity of \tAHD{} in $x$ from Theorem~\ref{theorem:angular depth continuity} then assert that the whole closed hemisphere $\Sph \cap H_{-u}$ is of minimum \tAHD{}. To prove the strict monotonicity of \tAHD{}, we pick any point $x \in \Sph$ from the \tAHD{}-median set of $P$. Because of the smoothness of $P$, we can argue as in \citet[Lemma~10]{Laketa_Nagy2022} to show that $\AHD_\alpha(P)$ must have non-empty (spherical) interior for all $\alpha \leq \max_{y \in \Sph}\AHD(y;P)$. Further, the dominant hemisphere condition says that for any $H, H' \in \half_0$ such that $H \cap H_u \subset H' \cap H_u$ we have
    \[  P(H') = P(H' \cap H) + P(H' \setminus H) > P(H' \cap H) + P(H \setminus H') = P(H), \]
since $H' \cap H \subset H_u$. This allows us to use the arguments as in the linear case and \tHD{} in parts (iii) and (v) of the proof of \citet[Theorem~9]{Laketa_Nagy2022}, which together give~\eqref{eq: strict monotonicity}. 

\subsection{Proof of Theorem~\ref{theorem:12}}

Our first claim can be proved analogously as for Theorem~\ref{theorem: uniform continuity}. The only exception is that in the final step, the last expression in~\eqref{eq: halfspace bound} vanishes as $n \to \infty$ almost surely, because the set of closed halfspaces $\half$ and its subset $\half_0$ are both Glivenko-Cantelli classes of sets \citep[Section 2.4 and Problem~14 in Section~2.6]{VanDerVaart_Wellner1996}.

For the second claim, part~\ref{theorem sample i} gives that~\eqref{eq: convergence condition} is valid for $\PP$-almost all $\omega \in \Omega$. Thus, we can separately apply Theorem~\ref{theorem: continuity of central regions} for each $\omega \in \Omega$ and obtain the result.

The last part of our theorem follows by direct application of part~\ref{theorem v} of Theorem~\ref{theorem:angular depth median continuity} and the theorem of Varadarajan \citep[see, e.g.,][Theorem~11.4.1]{Dudley2002}.

\subsection*{Acknowledgment}
The work of Stanislav Nagy was supported by Czech Science Foundation (EXPRO project n. 19-28231X).
 

\def\polhk#1{\setbox0=\hbox{#1}{\ooalign{\hidewidth
  \lower1.5ex\hbox{`}\hidewidth\crcr\unhbox0}}}

\end{document}